\documentclass[12pt, reqno]{amsart}
\usepackage{amssymb}
\usepackage{eucal}
\usepackage{amsmath}
\usepackage{amscd}
\usepackage[dvips]{color}
\usepackage{mathtools}
\usepackage{multicol}
\usepackage[all]{xy}           
\usepackage{graphicx}
\usepackage{color}
\usepackage{colordvi}
\usepackage{xspace}
\usepackage{ulem}
\usepackage{cancel}
\usepackage{tikz}
\usepackage{longtable}
\usepackage{multirow}
\usepackage{ifpdf}
\ifpdf
 \usepackage[colorlinks,final,backref=page,hyperindex]{hyperref}
\else
 \usepackage[colorlinks,final,backref=page,hyperindex,hypertex]{hyperref}
\fi

\begin{document}
\newcommand {\emptycomment}[1]{} 

\newcommand{\tabincell}[2]{\begin{tabular}{@{}#1@{}}#2\end{tabular}}

\newcommand{\nc}{\newcommand}
\newcommand{\delete}[1]{}

\nc{\mlabel}[1]{\label{#1}}  
\nc{\mcite}[1]{\cite{#1}}  
\nc{\mref}[1]{\ref{#1}}  
\nc{\meqref}[1]{\eqref{#1}} 
\nc{\mbibitem}[1]{\bibitem{#1}} 

\delete{
\nc{\mlabel}[1]{\label{#1}  
{\hfill \hspace{1cm}{\bf{{\ }\hfill(#1)}}}}
\nc{\mcite}[1]{\cite{#1}{{\bf{{\ }(#1)}}}}  
\nc{\mref}[1]{\ref{#1}{{\bf{{\ }(#1)}}}}  
\nc{\meqref}[1]{\eqref{#1}{{\bf{{\ }(#1)}}}} 
\nc{\mbibitem}[1]{\bibitem[\bf #1]{#1}} 
}

\newtheorem{thm}{Theorem}[section]
\newtheorem{lem}[thm]{Lemma}
\newtheorem{cor}[thm]{Corollary}
\newtheorem{pro}[thm]{Proposition}
\newtheorem{conj}[thm]{Conjecture}
\theoremstyle{definition}
\newtheorem{defi}[thm]{Definition}
\newtheorem{ex}[thm]{Example}
\newtheorem{rmk}[thm]{Remark}
\newtheorem{pdef}[thm]{Proposition-Definition}
\newtheorem{condition}[thm]{Condition}

\renewcommand{\labelenumi}{{\rm(\alph{enumi})}}
\renewcommand{\theenumi}{\alph{enumi}}
\renewcommand{\labelenumii}{{\rm(\roman{enumii})}}
\renewcommand{\theenumii}{\roman{enumii}}

\newcommand{\End}{\text{End}}
\newcommand{\gl}{\mathfrak{gl}}

\nc{\calb}{\mathcal{B}}
\nc{\call}{\mathcal{L}}
\nc{\calo}{\mathcal{O}}
\nc{\frakg}{\mathfrak{g}}
\nc{\frakh}{\mathfrak{h}}
\nc{\ad}{\mathrm{ad}}

\nc{\ldend}{L-dendriform algebra\xspace}
\nc{\ldends}{L-dendriform algebras\xspace}
\nc{\ldendb}{L-dendriform bialgebra\xspace}
\nc{\ldendbs}{L-dendriform bialgebras\xspace}
\nc{\spec}{special\xspace}
\nc{\nonsy}{nondegenerate symmetric\xspace}

\nc{\ccred}[1]{\tred{\textcircled{#1}}}

\nc{\move}[1]{{}}

\nc{\mrep}[2]{(#2,#1)}
\nc{\mgl}{\gl}

\nc{\rbrep}[3]{(#3,#1,#2)}

\nc{\copa}{\Delta}
\nc{\copb}{\nabla}
\nc{\copc}{\Diamond}



\title[Rota-Baxter Lie bialgebras and special L-dendriform bialgebras]
{Rota-Baxter Lie bialgebras, classical Yang-Baxter equations and special L-dendriform bialgebras
}

\author{Chengming Bai}
\address{Chern Institute of Mathematics \& LPMC, Nankai University, Tianjin 300071, China}
         \email{baicm@nankai.edu.cn}

\author{Li Guo}
\address{Department of Mathematics and Computer Science, Rutgers University, Newark, NJ 07102, USA}
         \email{liguo@rutgers.edu}

\author{Guilai Liu}
\address{Chern Institute of Mathematics \& LPMC, Nankai University, Tianjin 300071, China}
\email{1120190007@mail.nankai.edu.cn}

\author{Tianshui Ma}
\address{School of Mathematics and Information Science, Henan Normal University, Xinxiang 453007, China}
\email{matianshui@htu.edu.cn}

\date{\today}

\begin{abstract}
We establish a bialgebra structure on Rota-Baxter Lie algebras following the Manin triple approach to Lie bialgebras.
Explicitly, Rota-Baxter Lie bialgebras are characterized by generalizing matched pairs of Lie algebras and Manin triples of Lie algebras to the context of Rota-Baxter Lie algebras.
The coboundary case leads to the introduction of the admissible classical Yang-Baxter equation (CYBE) in Rota-Baxter Lie algebras, for which the antisymmetric solutions give rise to Rota-Baxter Lie bialgebras. The notions of $\mathcal{O}$-operators on Rota-Baxter Lie algebras and Rota-Baxter pre-Lie algebras are introduced to produce antisymmetric solutions of the admissible CYBE.
Furthermore, extending the well-known property that a Rota-Baxter
Lie algebra of weight zero induces a pre-Lie algebra,
the Rota-Baxter Lie bialgebra of weight zero induces
a bialgebra structure of independent interest, namely the \spec
\ldendb, which is equivalent to a Lie group with a left-invariant
flat pseudo-metric in geometry.
This induction
is also characterized as the inductions between the corresponding
Manin triples and matched pairs. Finally, antisymmetric solutions
of the admissible CYBE in a Rota-Baxter Lie algebra of weight zero
give \spec \ldendbs and in particular, both Rota-Baxter algebras
of weight zero and Rota-Baxter pre-Lie algebras of weight zero can
be used to construct \spec \ldendbs.
\end{abstract}

\subjclass[2010]{
17B38, 
17B62 
17B10, 
16T25, 
17A30, 
17A36,  
17D25.  
}

\keywords{Rota-Baxter operator; classical Yang-Baxter equation; pre-Lie algebra; bialgebra; special L-dendriform algebra}

\maketitle

\vspace{-1.3cm}

\tableofcontents

\allowdisplaybreaks

\section{Introduction}

This paper establishes a bialgebra structure on
Rota-Baxter Lie algebras following the approach of Manin triples in the classical work on Lie bialgebras~\mcite{Dri} and the recent development on Rota-Baxter antisymmetric infinitesimal bialgebras~\mcite{BGM}.
The well-known connection of Rota-Baxter Lie algebras with pre-Lie algebras is lifted to the bialgebra level, establishing a relationship of Rota-Baxter Lie bialgebras with \spec \ldendbs~\mcite{BHC}.

\vspace{-.2cm}

\subsection{Rota-Baxter Lie algebras}
The importance of Rota-Baxter Lie algebras can be viewed from
several perspectives.

\vspace{-.2cm}

\subsubsection{Rota-Baxter operators and Rota-Baxter associative algebras}

Let $A$ be a vector space equipped with a binary operation $\ast$ and let $\lambda$ be a scalar. A linear operator $R:A\to A$ is called a {\bf Rota-Baxter operator} of weight $\lambda$ if
\begin{equation} \mlabel{eq:rbo}
R(x)\ast R(y)=R(x\ast R(y))+R(R(x)\ast y)+\lambda R(x\ast y), \quad \forall x, y\in A.
\end{equation}
Then $(A,\ast,R)$, or simply $(A,R)$, is called a {\bf Rota-Baxter algebra} for the binary operation $\ast$, most notably, {\bf Rota-Baxter associative algebras} and {\bf Rota-Baxter Lie algebras}. 

The notion of Rota-Baxter associative algebra originated from the 1960 work~\mcite{Bax} of G.~Baxter
in a probability study, where it was noted that the identity is an
abstraction and generalization of the integration by parts
formula. Forty years later it
reappeared in the fundamental work of Connes and
Kreimer~\mcite{CK} on Hopf algebra approach to renormalization of
quantum field theory. Motivated by this and other connections in
combinatorics, number theory and operads, the study of Rota-Baxter
algebras has experienced a great expansion in the recent years.
See~\mcite{Guo2012,Rota1} for introductions and further
references.

\vspace{-.2cm}

\subsubsection{Rota-Baxter Lie algebras and the classical Yang-Baxter equation}

As a remarkable coincidence, Rota-Baxter operators on Lie algebras were discovered independently as the operator form of the classical Yang-Baxter equation (CYBE), named after the physicists.

The CYBE arose from the study of inverse scattering theory in the 1980s and was recognized as the
``semi-classical limit" of the quantum Yang-Baxter equation~\mcite{BaR,Ya}.
The study of the CYBE is also related to classical integrable
systems and quantum groups~\mcite{CP}.

An important approach in the study of the CYBE was through the
interpretation of its original tensor form in various operator forms which, along with the well-known method of Belavin and
Drinfeld~\mcite{BD}, has proved to be effective in providing solutions of the CYBE. Semonov-Tian-Shansky~\mcite{STS} first showed
that, if there exists a \nonsy invariant bilinear form on a Lie algebra $\frakg$, then a skew-symmetric solution $r$ of the CYBE can be equivalently expressed
as a linear operator $R:\frakg\to \frakg$ satisfying the operator identity
\begin{equation}\label{eq:CYBE-RB}
    [R(x), R(y)]=R([R(x),y])+R([x,R(y)]),\;\;\forall x,y\in
    \frakg,\end{equation}
which is then regarded as an {\bf operator form} of the CYBE.
Thus this operator form is simply the
Rota-Baxter relation (of weight zero) in Eq.~\meqref{eq:rbo} for Lie algebras.

This approach was expanded more generally by Kupershmidt~\mcite{Ku} by generalizing the notion of Rota-Baxter operators to $\calo$-operators, later also called relative Rota-Baxter operators and generalized Rota-Baxter operators~\mcite{PBG,Uc}.

\vspace{-.2cm}

\subsubsection{Rota-Baxter Lie algebras and pre-Lie algebras}

Another important role played by Rota-Baxter Lie algebras (of weight zero) is that they produce {\bf pre-Lie algebras}, defined to be a vector space $A$ with a binary operation $\circ$ satisfying
\begin{equation}\mlabel{eq:prelie}
    (x\circ y)\circ z-x\circ(y\circ z)=(y\circ x)\circ z-y\circ(x\circ z),\;\;\forall x,y,z\in A.
\end{equation}

Pre-Lie algebras, also called left-symmetric algebras,  originated from diverse areas of study, including convex
homogeneous cones \mcite{Vin}, affine manifolds and affine
structures on Lie groups \mcite{Ko}, deformation of associative
algebras \mcite{Ger}, and then appear in many more fields in
mathematics and mathematical physics, such as symplectic and
K\"{a}hler structures on Lie groups \mcite{Chu,Lic1988},  vertex
algebras \mcite{BK}, quantum field theory \mcite{CK} and operads
\mcite{CL}.
From the operadic viewpoint, the pre-Lie algebra is the splitting (successor) of the Lie algebra~\mcite{BBGN}. See \mcite{Bur} and the references
therein for more details.

By~\mcite{GS}, for a Rota-Baxter Lie algebra $(\mathfrak{g},[-,-],P)$ of weight zero, in defining
    \begin{equation}\mlabel{eq:pro:from invariance to left invariance1}
        x\circ y=[P(x),y],\;\;\forall x,y\in \mathfrak{g},
    \end{equation}
we obtain a pre-Lie algebra $(\frakg,\circ)$, called
the \textbf{induced pre-Lie algebra} from the Rota-Baxter Lie algebra $(\mathfrak{g},[-,-],P)$.

\vspace{-.3cm}

\subsubsection{Rota-Baxter Lie algebras and Lie
bialgebras}
The Lie bialgebra is the algebraic structure corresponding
to a Poisson-Lie group. It is also the classical structure of a quantized universal enveloping algebra~\mcite{CP,Dri}.
The great importance of Lie bialgebras is reflected by its
close relationship with several other fundamental notions.
First Lie bialgebras are
characterized by Manin triples and  matched pairs  of Lie algebras~\mcite{D1}.
In fact, there is a one-one correspondence between Lie
bialgebras and Manin triples of Lie algebras. The same holds for
Lie bialgebras and matched pairs  of Lie algebras associated to coadjoint representations.

Furthermore, antisymmetric solutions of the CYBE, or the  classical  $r$-matrices, naturally
give rise to coboundary Lie bialgebras \mcite{CP,STS}.
Furthermore, such solutions are provided by $\calo$-operators
which in turn are provided by pre-Lie algebras. In particular, as the $\mathcal O$-operators associated to adjoint representations,
Rota-Baxter operators provide antisymmetric  solutions of the CYBE in the bigger Lie algebras and hence give rise to Lie bialgebras.
See~\mcite{Bai2007,Bai2008,Bur,Ku} for more details.

As a summary, the following diagram illustrates the close
 relationship
of Rota-Baxter Lie algebras and $\calo$-operators with CYBE, pre-Lie algebras
 and Lie bialgebras. Here the correspondences going both ways are shown by arrows in
both directions, and the one-one correspondences are shown by
bi-directional double arrows.

\vspace{-.7cm}
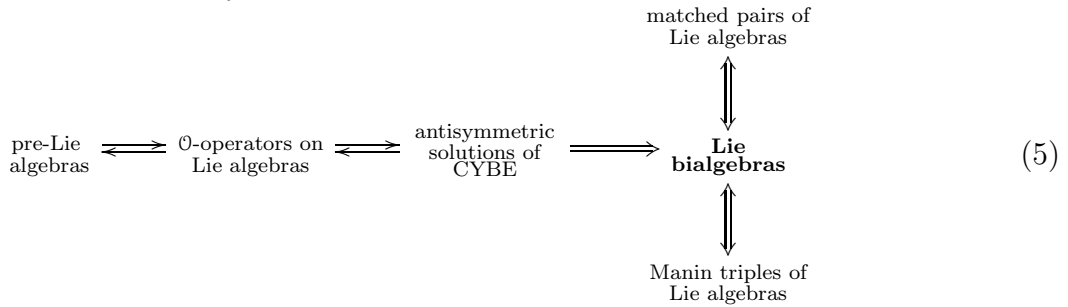
\begin{equation}
    \begin{split}
{\tiny        \xymatrix{
            &&&\txt{matched pairs of}\atop \txt{Lie algebras} \\
            \txt{pre-Lie}\atop\txt{ algebras} \ar@<.4ex>[r]      & {\txt{$\mathcal{O}$-operators on}} \atop\txt{Lie algebras}\ar@<.4ex>[l] \ar@<.4ex>[r]&
            {\txt{antisymmetric}\atop \txt{solutions of}}\atop \txt{CYBE} \ar@<.4ex>[l]
            \ar@2{->}[r]& \txt{\bf Lie}\atop \txt{\bf bialgebras} \ar@2{<->}[d] \ar@2{<->}[u] \\
            &&& \txt{Manin triples of}\atop \txt{Lie algebras} & }
}
  \end{split}
    \mlabel{eq:bigdiag}
\end{equation}
\vspace{-.6cm}

\subsection{Bialgebra structures on Rota-Baxter Lie algebras and \spec \ldends}

Given the importance of Rota-Baxter Lie algebras, especially its close relationship with CYBE and pre-Lie algebras, it is natural to consider the bialgebra structure on Rota-Baxter Lie algebras.

Similar to the quantized structures over
Lie algebras in terms of Lie bialgebras as Hopf algebras
(quantized universal enveloping algebras) and the quantization of
solutions of the CYBE as the solutions of the quantum Yang-Baxter
equation, one might expect that some quantized structures of
Rota-Baxter Lie algebras, as an operator form of the CYBE, can be
given by a bialgebra structure on Rota-Baxter Lie algebras and
shed further light on giving more solutions of the quantum
Yang-Baxter equation. While little is known about the
(associative) bialgebra structure of Rota-Baxter Lie algebras in
the universal enveloping algebras, Rota-Baxter Lie bialgebras constructed here might be regarded as an infinitesimal variation thereof.

Preferably, such a Rota-Baxter operator on Lie bialgebras should be part of a package of Rota-Baxter type actions applied systematically throughout the diagram in \meqref{eq:bigdiag}. This is what we will achieve, as depicted in the following diagram with the corresponding Lie bialgebra boldfaced in both diagrams.
\vspace{-.4cm}
    \begin{equation} \label{eq:rbdiag}
\begin{split}
{\tiny      \xymatrix{
                \txt{Rota-Baxter \\ pre-Lie algebras} \ar@<.4ex>[d]^-{\S\mref{sec:4.2}} & & \txt{Manin triples of \\ Rota-Baxter Lie algebras}
                \ar@{=>}[r]^-{\S\mref{sec:3.2}}_-{\lambda=0}
                \ar@{<=>}[d]^-{\S\mref{sec:3.1}} & \txt{Manin triples\\of pre-Lie algebras} \ar@{<=>}[d]^-{\S\mref{sec:3.2}}\\
                \txt{$\mathcal{O}$-operators on \\ Rota-Baxter \\ Lie algebras} \ar@<.4ex>[u] \ar@<.4ex>[r]^-{\S\mref{sec:4.2}} & \txt{antisymmetric\\solutions of \\ admissible CYBEs} \ar@<.4ex>[l] \ar@{=>}[r]^-{\S\mref{sec:4.1}} & \txt{\bf Rota-Baxter \\ \bf Lie bialgbras}
                \ar@{<=>}[d]^-{\S\mref{sec:3.1}}
                \ar@{=>}[r]^-{\S\mref{sec:3.2}}_-{\lambda=0}
                & \txt{special L-dendriform \\  bialgebras} \ar@{<=>}[d]^-{\S\mref{sec:3.2}}\\
                & & \txt{matched pairs of \\ Rota-Baxter Lie algebras}
                \ar@{=>}[r]^-{\S\mref{sec:3.2}}_-{\lambda=0}
                & \txt{matched pairs of \\ pre-Lie algebras}}
            }
\end{split}
\end{equation}

\vspace{-.2cm}

Thus in this paper, we first
establish a bialgebra theory for Rota-Baxter Lie algebras
following the Manin triple approach of Lie bialgebras~\mcite{BD}
and Rota-Baxter ASI bialgebras given in \mcite{BGM}. Most of the
results in \mcite{BGM} for Rota-Baxter associative algebras remain
valid for Rota-Baxter Lie algebras. Therefore by adding the roles
of Rota-Baxter operators in diagram~(\ref{eq:bigdiag}), we get the
left part of the diagram~(\ref{eq:rbdiag}).

For a Lie bialgebra given by a Manin triple $(\frakg\oplus \frakg^*, \frakg, \frakg^*)$ together with a Rota-Baxter operator $P$ on $\frakg$, in building a Rota-Baxter Lie bialgebra on top of this, we allow freedom on a Rota-Baxter operator $Q^*$ on the Lie algebra $\frakg^*$, as long as $P$ and $Q^*$ satisfy certain admissibility condition. This general approach allows us to recover the previous constructions as special cases: the Rota-Baxter Lie bialgebra independently developed in~\mcite{LS} corresponds to the case when $Q=-P-{\rm id}$ and the weight $\lambda=1$, while the Rota-Baxter Lie
bialgebra introduced in~\mcite{Shi} corresponds to the case when
$Q=-P$ and $\lambda=0$.

Another interesting phenomenon in our construction is reflected in the induced structures from Rota-Baxter Lie bialgebras. Since a Rota-Baxter Lie algebra induces a pre-Lie algebra by Eq.~\meqref{eq:pro:from invariance to left invariance1}, it is natural to expect that a Rota-Baxter Lie bialgebra induces a pre-Lie bialgebra as defined in~\mcite{Bai2008}. As it turns out, this is not the case.
To determine the correct induced structure from a Rota-Baxter Lie bialgebra, we regard the latter as a Lie bialgebra with a Rota-Baxter operator. Since a Lie bialgebra is characterized by a matched pair or a Manin triple of Lie algebras, as shown in the right column in the diagram~\meqref{eq:bigdiag}, we can also regard a Rota-Baxter Lie bialgebra as a matched pair or Manin triple of Lie algebras equipped with a Rota-Baxter operator. Since a Rota-Baxter operator on a Lie algebra induces a pre-Lie algebra, a Rota-Baxter operator on a matched pair or Manin triple of Lie algebras should induce a matched pair or Manin triple of pre-Lie algebras. Quite unexpectedly to us, the resulting bialgebra structure is not a pre-Lie bialgebra as introduced in~\mcite{Bai2008}, but a \spec \ldendb~\mcite{BHC}, leading to the commutative double square diagram in~\meqref{eq:rbdiag}.

As it turns out, the class of \spec \ldendb is quite interesting on its own right as a subclass of \ldendbs~\mcite{NB} and has already been studied in some depth.
The operad of L-dendriform algebras~\mcite{BLN} is the two-fold splitting (successor) of the operad of Lie algebras (see Remark~\mref{rk:ldend}). Thus the correspondences in the right column of the diagram in \meqref{eq:rbdiag} suggest that the splitting of a bialgebra structure is the bialgebra of the two-fold splitting of the structure. Moreover, the \spec \ldendb has a clear and important geometric interpretation. Since the study of
Lie groups with a left-invariant flat pseudo-metric is equivalent to the study of the pre-Lie algebras with a \nonsy left-invariant bilinear form~\mcite{Mil}, the \spec \ldendbs correspond to a class of Lie groups with a left-invariant flat pseudo-metric.

\vspace{-.3cm}

\subsection{Layout of the paper}
As outlined in the diagram~\meqref{eq:rbdiag}, the paper is organized as follows.

In Section \mref{sec:2}, we give the notion of representations of
Rota-Baxter Lie algebras. An admissibility of a
linear operator for a Rota-Baxter Lie algebra is introduced in order to construct a reasonable representation on the dual space. We also observe that an invariant bilinear form on a
Rota-Baxter Lie algebra of weight zero is left-invariant on the
induced pre-Lie algebra. Moreover, we interpret special
\ldends in terms of the representations of pre-Lie
algebras and hence a Rota-Baxter Lie algebra of weight zero with a
linear operator satisfying the aforementioned admissibility
condition gives a \spec \ldend which is compatible
with the induced pre-Lie algebra.

In Section \mref{sec:3}, we introduce the notion of Rota-Baxter Lie bialgebras, equivalently characterized by Manin triples of Rota-Baxter Lie algebras and matched pairs of Rota-Baxter Lie algebras. We establish the explicit relationship between Rota-Baxter Lie  bialgebras of weight zero and \spec \ldendbs, previously established in \mcite{BHC}, both directly and in their respective equivalent interpretations in terms of the corresponding Manin triples and matched pairs.

In Section \mref{sec:4}, we focus on coboundary Rota-Baxter Lie
bialgebras, leading to the introduction of the admissible CYBE in
Rota-Baxter Lie algebras whose antisymmetric solutions are used
to construct Rota-Baxter Lie bialgebras.
The notions of $\mathcal{O}$-operators on Rota-Baxter Lie algebras and Rota-Baxter pre-Lie algebras,
are introduced to produce antisymmetric solutions of the admissible
CYBE and hence Rota-Baxter Lie bialgebras. Furthermore, when the weight is zero, we study the induced \spec \ldendbs from these Rota-Baxter Lie bialgebras, and thus give the construction of special L-dendriform bialgebras from antisymmetric solutions of the admissible CYBE in Rota-Baxter Lie
algebras of weight zero.  In particular, both Rota-Baxter Lie
algebras of weight zero and Rota-Baxter pre-Lie algebras of weight
zero give rise to \spec \ldendbs.

\noindent
{\bf Notations.}
Unless otherwise specified, all the vector spaces and algebras are finite dimensional over a field $\mathbb {K}$ of characteristic zero, although many results and notions, in particular that of a Rota-Baxter Lie bialgebra, remain valid in the infinite-dimensional case.

\vspace{-.3cm}

\section{Rota-Baxter Lie algebras, pre-Lie algebras and \spec \ldends}\mlabel{sec:2}
We give the notion of representations of Rota-Baxter Lie algebras.
An admissibility of a linear map for a Rota-Baxter
Lie algebra is introduced in order to obtain a reasonable
representation on the dual space. On the other hand, we observe
that an invariant bilinear form on a  Rota-Baxter Lie algebra of
weight zero is left-invariant on the induced pre-Lie algebra.
Moreover, we interpret \spec \ldends in terms of
the representations of pre-Lie algebras. Hence a Rota-Baxter
Lie algebra of weight zero with a linear operator satisfying the
aforementioned admissibility condition gives a special
\ldend.

\subsection{Rota-Baxter Lie algebras and their representations}\mlabel{sec:2.1}\

We first recall some basic facts on the representations of Lie algebras.
For a Lie algebra $\frakg\coloneqq (\frakg,[-,-])$, a \textbf{representation} of $\frakg$ is a pair $\mrep{\rho}{V}$
consisting of a vector space $V$ and a Lie algebra homomorphism
\vspace{-.2cm}
$$\rho:\mathfrak{g}\rightarrow \mgl(V),$$
for the natural Lie algebra structure on $\mgl(V)=\End(V)$.

In particular, the linear map
$$\mathrm{ad}:\mathfrak{g}\rightarrow \mgl(\mathfrak{g}), \quad \mathrm{ad}(x)y=[x,y],\quad \forall x,y\in\mathfrak{g},$$
defines a representation $\mrep{\mathrm{ad}}{\mathfrak{g}}$ of $(\mathfrak{g},[-,-])$, called the \textbf{adjoint representation}.

For a vector space $V$ and a linear map $\rho:\frakg\to \mgl(V)$, the pair $\mrep{\rho}{V}$ is a representation of $(\mathfrak{g},[-,-])$ if and only if the
operation $[-,-]_{\mathfrak{g}\oplus V}$ (often still denoted by $[-,-]$ for simplicity) on $\mathfrak{g}\oplus V$ defined by
\begin{equation}\mlabel{eq:pro:SD RB Lie1}
[x+u,y+v]_{\mathfrak{g}\oplus V}\coloneqq [x,y]+\rho(x)v-\rho(y)u,\;\;\forall x,y\in\mathfrak{g}, u,v\in V,
\end{equation}
makes $\mathfrak{g}\oplus V$ into a Lie algebra, called the {\bf semi-direct product} of $(\frak g,[-,-])$ by $V$, and denoted by $\frak g\ltimes_\rho V$.

Let $A$ and $V$ be vector spaces. For a linear map $\rho:A\rightarrow\End(V)$, we set $\rho^{*}:A\rightarrow\mathrm{End}(V^{*})$ by
\begin{equation}
\langle\rho^{*}(x)v^{*},u\rangle=-\langle v^{*},\rho(x)u\rangle,\;\;\forall x\in A, u\in V, v^{*}\in V^{*}.
\end{equation}
Here $\langle\ ,\ \rangle$ is the usual pairing between $V$ and $V^*$.
If $\mrep{\rho}{V}$ is a representation of a Lie algebra $(\mathfrak{g},[-,-])$, then $\mrep{\rho^{*}}{V^{*}}$ is also a representation of $(\mathfrak{g},[-,-])$, called the {\bf dual representation} of $\mrep{\rho}{V}$.
In particular, $\mrep{\mathrm{ad}^{*}}{\mathfrak{g}^{*}}$ is a representation of $(\mathfrak{g},[-,-])$.

When we extend these notions to Rota-Baxter Lie algebras next, special attention needs to be given to the dual representations.

\begin{defi}\mlabel{defi:rep RB}
A \textbf{representation} of a Rota-Baxter Lie algebra $(\mathfrak{g},[-,-],P)$ is a triple $\rbrep{\rho}{\alpha}{V}$,
such that $\mrep{\rho}{V}$ is a representation of the Lie algebra $(\mathfrak{g},[-,-])$ and  $\alpha:V\rightarrow V$ is a linear map satisfying the following equation:
\begin{equation}\mlabel{eq:defi:rep RB1}
\rho(P(x))\alpha(v)=\alpha(\rho(P(x))v)+\alpha(\rho(x)\alpha(v))+\lambda\alpha(\rho(x)v),\;\;\forall x\in \mathfrak{g},v\in V.
\end{equation}
Two representations
$\rbrep{\rho_{1}}{\alpha_{1}}{V_{1}}$ and $\rbrep{\rho_{2}}{\alpha_{2}}{V_{2}}$
of a Rota-Baxter Lie algebra $(\mathfrak{g},[-,-],P)$ are called \textbf{equivalent} if there exists a linear isomorphism $\varphi:V_{1}\rightarrow V_{2}$ such that
\begin{equation}\mlabel{eq:defi:rep RB2}
\varphi(\rho_{1}(x)v)=\rho_{2}(x)\varphi(v),\quad  \varphi(\alpha_{1}(v))=\alpha_{2}(\varphi(v)),\;\;\forall x\in\mathfrak{g}, v\in V_{1}.
\end{equation}
\end{defi}

\begin{ex}
    Let $(\mathfrak{g},[-,-],P)$ be a Rota-Baxter Lie algebra of weight $\lambda$.
    \begin{enumerate}
        \item
        $\rbrep{\mathrm{ad}}{P}{\mathfrak{g}}$ is a representation of $(\mathfrak{g},[-,-],P)$, called the \textbf{adjoint} representation of $(\mathfrak{g}$,\
        $[-,-],P)$.
        \item Let $\mrep{\rho}{V}$ be a representation of the Lie algebra $(\mathfrak{g},[-,-])$. Then
        $\rbrep{\rho}{0}{V}$ and $\rbrep{\rho}{-\lambda\mathrm{id}_V}{V}$ are representations of $(\mathfrak{g},[-,-],P)$.
    \end{enumerate}
\end{ex}

For vector spaces $V_1$ and $V_2$ and linear maps $\phi_1:V_1\rightarrow V_1$ and $\phi_2:V_2\rightarrow V_2$, let $\phi_1+\phi_2$ denote the linear map:
\begin{equation}\mlabel{eq:pro:SD RB Lie2}
\phi_{V_1\oplus V_2}:V_1\oplus V_2\rightarrow V_1\oplus V_2,\;\;\phi_{V_1\oplus V_2}(v_1+v_2)=\phi_1(v_1)+\phi_2(v_2),\;\;\forall v_1\in V_1,v_2\in V_2.
\end{equation}

The representations of Rota-Baxter Lie algebras are characterized as follows.

\begin{pro}\mlabel{pro:SD RB Lie}
Let $(\mathfrak{g},[-,-],P)$ be a Rota-Baxter Lie algebra of weight $\lambda$ and $V$ be a vector space. Let $\rho:\mathfrak{g}\rightarrow\mathrm{End}(V)$ and $\alpha:V\rightarrow V$ be linear maps. Then $\rbrep{\rho}{\alpha}{V}$ is a representation of $(\mathfrak{g},[-,-],P)$ if and only if, for the
Lie algebra $(\frakg\oplus V, [-,-]_{\frakg\oplus V})$ defined in Eq.~\meqref{eq:pro:SD RB Lie1} and the linear map $P+\alpha$ defined in Eq.~\meqref{eq:pro:SD RB Lie2}, the triple
 $(\mathfrak{g}\oplus V,[-,-]_{\frakg\oplus V}, P+\alpha)$ is a Rota-Baxter Lie algebra of weight $\lambda$.
In this case, the resulting Rota-Baxter Lie algebra is denoted by $(\mathfrak{g}\ltimes_{\rho}V,[-,-]_{\frakg\oplus V}, P+\alpha)$ and called the \textbf{semi-direct product} of $(\mathfrak{g},[-,-],P)$ by $\rbrep{\rho}{\alpha}{V}$.
\end{pro}

The proof is omitted since the proposition is a special case of the matched pairs of Rota-Baxter Lie algebras in Proposition~\mref{pro:MP RB}, when $B=V$ is equipped with the zero operation.

For a representation of a Rota-Baxter Lie algebra, in order to obtain a representation of the Rota-Baxter Lie algebra on the dual space, an additional condition is needed.

Let $V$ and $W$ be vector spaces.
For a linear map $T:V\rightarrow W$, the transpose map $T^{*}:W^{*}\rightarrow V^{*}$ is characterized by
$$\langle T^{*}(w^{*}),v\rangle=\langle w^{*},T(v)\rangle,\;\;\forall v\in V, w^* \in W^{*}.$$

\begin{lem}\mlabel{lem:dual map}
Let $(\mathfrak{g},[-,-],P)$ be a  Rota-Baxter Lie algebra of weight $\lambda$,
    $\mrep{\rho}{V}$ be a representation of $(\mathfrak{g},[-,-])$, and $\beta:V\rightarrow V$ be a linear map.
    Then
    $\rbrep{\rho^{*}}{\beta^{*}}{V^{*}}$
    is a representation of  $(\mathfrak{g},[-,-],P)$ if and only if
    \begin{equation}\mlabel{eq:lem:dual map1}
        \beta(\rho(P(x))v)-\rho(P(x))\beta(v)-\beta(\rho(x)\beta(v))-\lambda\rho(x)\beta(v)=0,\;\;\forall x\in\mathfrak{g}, v\in V.
    \end{equation}
In particular, for a linear map $Q:\mathfrak{g}\rightarrow \mathfrak{g}$, the triple
$\rbrep{\mathrm{ad}^{*}}{Q^{*}}{\mathfrak{g}^{*}}$
is a representation of
$(\mathfrak{g},[-,-],P)$ if and only if
\begin{equation}\mlabel{eq:cor:adm1}
    Q([P(x),y])-[P(x),Q(y)]-Q([x,Q(y)])-\lambda[x,Q(y)]=0,\;\;\forall x,y\in\mathfrak{g}.
\end{equation}
\end{lem}
\begin{proof}
The proof of the first claim is similar to that of  \cite[Lemma 2.11]{BGM}. The second claim follows from the first by taking the adjoint representation.
\end{proof}

\begin{defi}\mlabel{defi:admissibility}
Let $(\mathfrak{g},[-,-],P)$ be a  Rota-Baxter Lie algebra of weight $\lambda$,
$\mrep{\rho}{V}$ be a representation of $(\mathfrak{g},[-,-])$, and $\beta:V\rightarrow V$ be a linear map. If
$\rbrep{\rho^{*}}{\beta^{*}}{V^{*}}$
is a representation of  $(\mathfrak{g},[-,-],P)$ (that is, Eq.~\meqref{eq:lem:dual map1} holds), then we say that $\beta$ is \textbf{admissible to the Rota-Baxter Lie algebra} $(\mathfrak{g},[-,-],P)$ \textbf{on} $\mrep{\rho}{V}$, or $(\mathfrak{g},[-,-],P)$ is $\beta$-\textbf{admissible} \textbf{on} $\mrep{\rho}{V}$. In particular, if there is a linear map $Q:\mathfrak{g}\rightarrow\mathfrak{g}$ satisfying Eq.~\meqref{eq:cor:adm1}  such that
$\rbrep{\mathrm{ad}^{*}}{Q^{*}}{\mathfrak{g}^{*}}$
is a representation of
$(\mathfrak{g},[-,-],P)$,
we simply say that \textbf{$Q$ is admissible to $(\mathfrak{g},[-,-],P)$} or  \textbf{$(\mathfrak{g},[-,-],P)$ is $Q$-admissible}.
\end{defi}

\begin{pro}\mlabel{pro:admissibility}
    Let $(\mathfrak{g},[-,-],P)$ be a Rota-Baxter Lie algebra of weight $\lambda$ and $\rbrep{\rho}{\alpha}{V}$ be a representation of $(\mathfrak{g},[-,-],P)$. Then $-\alpha-\lambda\mathrm{id}_{V}$ is admissible to $(\mathfrak{g},[-,-],P)$ on $\mrep{\rho}{V}$. In particular,  for any representation $\mrep{\rho}{V}$ of the Lie algebra $(\frak g,[-,-])$, both $-\lambda\mathrm{id}_{V}$ and 0 are admissible to $(\mathfrak{g},[-,-],P)$ on $\mrep{\rho}{V}$.
\end{pro}

\begin{proof}
Note that Eq.~\meqref{eq:lem:dual map1} holds by taking $\beta=-\alpha-\lambda\mathrm{id}_{V}$. Hence the first conclusion follows.
The other conclusions follow by taking $\alpha=0$ and $\alpha=-\lambda {\rm id}_V$ respectively.
\end{proof}

Taking the adjoint representation in Proposition \mref{pro:admissibility}, we obtain

\begin{cor}\mlabel{cor:adm}
    Let $(\mathfrak{g},[-,-],P)$ be a Rota-Baxter Lie algebra of weight $\lambda$. Then the linear maps $-P-\lambda\mathrm{id}_{\mathfrak{g}}$,
 $-\lambda\mathrm{id}_{\mathfrak{g}}$  and $0$ are admissible to $(\mathfrak{g},[-,-],P)$.
\end{cor}

We next give admissible representations by invariant bilinear forms on Lie algebras.
Recall that a bilinear form $\mathcal{B}$ on a Lie algebra $(\mathfrak{g},[-,-])$ is called \textbf{invariant} if
\begin{equation}\mlabel{eq:defi:invariance on Lie}
\mathcal{B}([x,y],z)=\mathcal{B}(x,[y,z]),\;\;\forall x,y,z\in \mathfrak{g}.
\end{equation}

\begin{pro}\mlabel{pro:invariance}
Let $(\mathfrak{g},[-,-],P)$ be a Rota-Baxter Lie algebra and
$\mathcal{B}$ be a nondegenerate invariant bilinear form on the Lie algebra $(\frak g, [-,-])$.
Let $\widehat{P}$ be the adjoint linear map of $P$ with respect to $\mathcal{B}$, that is, $\widehat{P}$ is characterized by
\begin{equation} \mlabel{eq:adop}
\mathcal{B}(P(x),y)=\mathcal{B}(x,\widehat{P}(y)),\quad  \forall x,y\in\mathfrak{g}.
\end{equation}
Then $\widehat{P}$ is admissible to $(\mathfrak{g},[-,-],P)$, or equivalently,
$\rbrep{\mathrm{ad}^{*}}{\widehat{P}^{*}}{\mathfrak{g}^{*}}$
is a representation of $(\mathfrak{g}$, $[-,-]$, $P)$.  Moreover,  $\rbrep{\mathrm{ad}^{*}}{\widehat{P}^{*}}{\mathfrak{g}^{*}}$  is equivalent to $\rbrep{\mathrm{ad}}{P}{\mathfrak{g}}$ as a representation of $(\mathfrak{g},[-,-],P)$.

Conversely, let $(\mathfrak{g},[-,-],P)$ be a Rota-Baxter Lie algebra and $Q:\mathfrak{g}\rightarrow\mathfrak{g}$ be a linear map that is admissible to $(\mathfrak{g},[-,-],P)$. If the resulting representation $\rbrep{\mathrm{ad}^{*}}{Q^{*}}{\mathfrak{g}^{*}}$  of $(\mathfrak{g},[-,-],P)$ is equivalent to $\rbrep{\mathrm{ad}}{P}{\mathfrak{g}}$, then there exists a nondegenerate invariant bilinear from $\mathcal{B}$ on $(\mathfrak{g},[-,-],P)$ such that $\widehat{Q}=P$.
\end{pro}
\begin{proof}
The proof is similar to that of \cite[Proposition 3.9]{BGM}. Note that in the context of Lie algebras,  by the antisymmetry of the Lie bracket, the bilinear form $\mathcal{B}$ no longer needs to be symmetric.
\end{proof}

\subsection{Pre-Lie algebras and special L-dendriform
algebras}\mlabel{sec:2.2}\
We first recall some facts on pre-Lie algebras \mcite{Bur}. For a pre-Lie algebra $(A,\circ)$ as defined in Eq.~\meqref{eq:prelie}, the commutator
\begin{equation}[x,y]=x\circ y-y\circ x,\;\; \forall x,y\in A,
 \end{equation}
 defines a Lie algebra $(\mathfrak{g}(A),[-,-])$, called the \textbf{sub-adjacent} Lie algebra of $(A,\circ)$, and $(A,\circ)$ is called a \textbf{compatible} pre-Lie algebra structure on the Lie algebra $(\mathfrak{g}(A),[-,-])$.

 For a vector space $A$ with a binary operation $\circ:A\otimes A\rightarrow A$, define linear maps $$\mathcal{L}_{\circ}, \mathcal{R}_{\circ}:A\rightarrow\mathrm{End}(A), \quad \mathcal{L}_{\circ}(x)y\coloneqq x\circ y\eqqcolon \mathcal{R}_{\circ}(y)x, \quad \forall x,y\in A.$$
 Then for a pre-Lie algebra $(A,\circ),$
$\mrep{\mathcal{L}_{\circ}}{A}$ is a representation of the sub-adjacent Lie algebra $(\mathfrak{g}(A),[-,-])$.

\begin{defi}
A bilinear form $\mathcal{B}$ on a pre-Lie algebra $(A,\circ)$ is called \textbf{left-invariant} if
\begin{equation}\mlabel{eq:defi:left invariance on pre-Lie}
    \mathcal{B}(x\circ y,z)+\mathcal{B}(y, x\circ z)=0,\;\;\forall x,y,z\in A.
\end{equation}
\end{defi}

\begin{rmk}
The references \mcite{Aub,Mil} give a natural bijection between the set of the pre-Lie algebras with a \nonsy left-invariant
bilinear form and the set of the connected and simply-connected
Lie groups with a left-invariant flat pseudo-metric. Under correspondence, the sub-adjacent Lie algebra of a pre-Lie algebra in the former set is precisely the Lie algebra of the corresponding Lie group.
\end{rmk}

\begin{pro}\mlabel{pro:from invariance to left invariance}
    Let $(\mathfrak{g},[-,-],P)$ be a  Rota-Baxter Lie algebra of weight zero and $(\frak g,\circ)$ be the induced pre-Lie algebra.
If there is an invariant bilinear form $\mathcal{B}$ on the Lie algebra $(\mathfrak{g},[-,-])$, then $\mathcal{B}$ is left-invariant on the pre-Lie algebra $(\mathfrak{g},\circ)$.
\end{pro}

\begin{proof}
For all $x,y,z\in \mathfrak{g}$, we have
$$\mathcal{B}(x\circ y,z)=\mathcal{B}([P(x),y],z)=-\mathcal{B}(y,[P(x),z])=-\mathcal{B}(y,x\circ z).$$
Hence the conclusion holds.
\end{proof}

\begin{defi}\mlabel{defi:rep of pre-Lie}
A \textbf{representation} of a pre-Lie algebra $(A,\circ)$ is a triple
$\rbrep{l_{\circ}}{r_{\circ}}{V}$,
where $V$ is a vector space, and $l_{\circ}, r_{\circ}:A\rightarrow\mathrm{End}(V)$ are linear maps satisfying
\begin{equation}\mlabel{eq:defi:rep of pre-Lie1}
l_{\circ}(x)l_{\circ}(y)v-l_{\circ}(x\circ y)v=l_{\circ}(y)l_{\circ}(x)v-l_{\circ}(y\circ x)v,
\end{equation}
\begin{equation}\mlabel{eq:defi:rep of pre-Lie2}
l_{\circ}(x)r_{\circ}(y)v-r_{\circ}(y)l_{\circ}(x)v=r_{\circ}(x\circ y)v-r_{\circ}(y)r_{\circ}(x)v,
\end{equation}
for all $x,y\in A, v\in V$.
\end{defi}

In fact, $\rbrep{l_{\circ}}{r_{\circ}}{V}$ is a representation of a pre-Lie algebra $(A,\circ)$ if and only if the direct sum $A\oplus V$ of vector spaces is equipped with a ({\bf semi-direct product}) pre-Lie algebra structure by the multiplication on $A\oplus V$ defined by
\begin{equation}
(x +u)\circ_{A\oplus V}(y +v ):=x \circ y+l_{\circ}(x)v+r_{\circ}(y)u,\;\;\forall x, y\in A, u, v\in V.
\end{equation}
We denote the resulting pre-Lie algebra by $A\ltimes_{l_{\circ}, r_{\circ}}V$ or simply $A\ltimes V$. A representation of a pre-Lie algebra also has a naturally defined dual representation.

\begin{lem}\mlabel{lem:dual rep}
\mcite{Bai2008}
    Let $\rbrep{l_{\circ}}{r_{\circ}}{V}$ be a representation of a pre-Lie algebra $(A,\circ)$. Then
    $\rbrep{l^{*}_{\circ}-r^{*}_{\circ}}{r_{\circ}^{*}}{V^{*}}$
    is a representation of $(A,\circ)$.
\end{lem}

A representation of a Rota-Baxter Lie algebra gives rise to a representation of the induced pre-Lie algebra.

\begin{pro}\mlabel{pro:pre-Lie rep1}
    Let $\rbrep{\rho}{\alpha}{V}$ be a representation of a Rota-Baxter Lie algebra $(\mathfrak{g},[-,-]$, $P)$ of weight zero. Define $l_{\rho,\alpha}, r_{\rho,\alpha}:\frakg\to \End(V)$ by
\begin{equation}
        l_{\rho,\alpha}(x)v=\rho(P(x))v, \ r_{\rho,\alpha}(x)v=-\rho(x)\alpha(v),\;\;\forall x\in\mathfrak{g}, v\in V.
    \end{equation}
Then
$\rbrep{l_{\rho,\alpha}}{r_{\rho,\alpha}}{V}$
is a representation of the induced pre-Lie algebra $(\mathfrak{g},\circ)$,  called the {\bf induced representation} of the pre-Lie algebra $(\frak g,\circ)$ from $\rbrep{\rho}{\alpha}{V}$.
\end{pro}

\begin{proof}
By Proposition~\mref{pro:SD RB Lie}, there is a Rota-Baxter Lie algebra $(\mathfrak{g}\ltimes_{\rho}V,P+\alpha)$ of weight zero. Hence by
Eq.~\meqref{eq:pro:from invariance to left invariance1}, there is an induced pre-Lie algebra structure on $\mathfrak{g}\oplus V$, defined by
\vspace{-.2cm}
    \begin{eqnarray*}
        (x+u)\circ(y+v)&=&[(P+\alpha)(x+u),y+v]=[P(x),y]+\rho(P(x))v-\rho(y)\alpha(u)\\
        &=&x\circ y+l_{\rho,\alpha}(x)v+r_{\rho,\alpha}(y)u,
    \end{eqnarray*}
\vspace{-.2cm}
    for all $x,y\in\mathfrak{g}, u,v\in V$. Thus $\rbrep{l_{\rho,\alpha}}{r_{\rho,\alpha}}{V}$ is a representation of $(\mathfrak{g},\circ)$.
\end{proof}

We next recall the notions of \ldends and \spec \ldends.

\begin{defi}\mlabel{defi:L-dend}\mcite{BHC,BLN}
An \textbf{\ldend} is a triple $(A,\triangleright,\triangleleft)$, such that $A$ is a vector space, and $\triangleright,\triangleleft:A\otimes A\rightarrow A$ are linear maps satisfying
\begin{eqnarray}\mlabel{eq:defi:L-dend1}
&(x\triangleright y)\triangleright z+(x\triangleleft y)\triangleright z+y\triangleright(x\triangleright z)-(y\triangleleft x)\triangleright z-(y\triangleright x)\triangleright z-x\triangleright(y\triangleright z)=0,& \\
\mlabel{eq:defi:L-dend2}
&(x\triangleright y)\triangleleft z+y\triangleleft(x\triangleright z)+y\triangleleft(x\triangleleft z)-(y\triangleleft x)\triangleleft z-x\triangleright(y\triangleleft z)=0,&
\end{eqnarray}
for all $x,y,z\in A$. An \ldend is called \textbf{special} if $\triangleleft$ is antisymmetric.
\end{defi}

\begin{rmk} \mlabel{rk:ldend}
\begin{enumerate}
\item The operad of \ldends is the successor~\mcite{BBGN} of the
operad $\text{\bf pre-Lie}$ of pre-Lie algebras. Thus it is also
the Manin black product $\text{\bf pre-Lie}\bullet \text{\bf
pre-Lie}$~\cite[Corollary~3.5]{BBGN}. 
\item A Rota-Baxter operator
$P$ of weight zero on a pre-Lie algebra $(A,\circ)$ gives rise to
an \ldend with the  multiplications~\cite[Corollary~3.9]{BLN}
$$ x\triangleright y\coloneqq P(x)\circ y, \quad x\triangleleft y\coloneqq -y\circ P(x), \quad \forall x, y\in A.$$
In a similar manner, a commuting pair of Rota-Baxter operators of
weight zero on a Lie algebra gives rise to an
\ldend~\cite[Corollary~3.10]{BLN}.
\end{enumerate}
\end{rmk}

Furthermore~\mcite{BHC,BLN}, for an \ldend $(A,\triangleright,\triangleleft)$, there are pre-Lie algebras $(A,\circ)$ and $(A,\star)$ given by
\begin{equation}\mlabel{eq:horizontal and vertical}
x\circ y=x\triangleright y-y\triangleleft x, \quad  x\star y=x\triangleright y+x\triangleleft y,\;\;\forall x,y\in A,
\end{equation}
called the \textbf{horizontal} and \textbf{vertical} pre-Lie algebras respectively. Moreover, if (and only if under our assumption of characteristic zero) $(A,\triangleright,\triangleleft)$ is special, then the horizontal pre-Lie algebra $(A,\circ)$ and the vertical pre-Lie algebra $(A,\star)$ coincide, that is, $x\circ y=x\star y, \forall x,y\in A$. In this case, $(A,\circ)$ is called the \textbf{sub-adjacent} pre-Lie algebra of $(A,\triangleright,\triangleleft)$, and $(A,\triangleright,\triangleleft)$ is called a \textbf{compatible} \spec \ldend of $(A,\circ)$.

We now interpret \spec \ldends in terms of representations of pre-Lie algebras.

\begin{pro}\mlabel{pro:equivalence}
Let $(A,\circ)$ be a pre-Lie algebra. Suppose that $\triangleleft:A\otimes A\rightarrow A$ is an antisymmetric multiplication on $A$. Define a multiplication $\triangleright$ on $A$ by
\begin{equation}x\triangleright y\coloneqq x\circ y-x\triangleleft y, \quad \forall x,y\in A.\mlabel{eq:right}\end{equation}
Then the following statements are equivalent.
\begin{enumerate}
\item $(A,\triangleright,\triangleleft)$ is a \spec \ldend;
\mlabel{it:equiv1}
\item \mlabel{it:equiv2}
The following equation holds:
\begin{equation}\mlabel{eq:skew-sym op}
x\triangleleft(y\triangleleft z)+y\triangleleft(x\circ z)-z\triangleleft(x\circ y)-x\circ(y\triangleleft z)=0,\;\;\forall x,y,z\in A;
\end{equation}

\item \mlabel{it:equiv3}
$\rbrep{\mathcal{L}_{\circ}-\mathcal{L}_{\triangleleft}}{-\mathcal{L}_{\triangleleft}}{A}$
is a representation of $(A,\circ)$;

\item \mlabel{it:equiv4}
$\rbrep{\mathcal{L}^{*}_{\circ}}{\mathcal{L}^{*}_{\triangleleft}}{A^{*}}$
is a representation of $(A,\circ)$.
\end{enumerate}
\end{pro}
\begin{proof}
$\meqref{it:equiv1} \Longleftrightarrow \meqref{it:equiv2}$. Let $x,y,z\in A$. Then we have
\begin{eqnarray*}
&&(x\triangleright y)\triangleleft z+y\triangleleft(x\triangleright z)+y\triangleleft(x\triangleleft z)-(y\triangleleft x)\triangleleft z-x\triangleright(y\triangleleft z)\\
&&=(x\circ y)\triangleleft z+y\triangleleft(x\circ z)-x\circ(y\triangleleft z)+x\triangleleft(y\triangleleft z).
\end{eqnarray*}
Thus Eq.~\meqref{eq:defi:L-dend2} holds if and only if Eq.~\meqref{eq:skew-sym op} holds. Moreover, if this is the case, then we have
\begin{eqnarray*}
&&(x\triangleright y)\triangleright z+(x\triangleleft y)\triangleright z+y\triangleright(x\triangleright z)-(y\triangleleft x)\triangleright z-(y\triangleright x)\triangleright z-x\triangleright(y\triangleright z)\\
&&=(x\circ y)\circ z-(x\circ y)\triangleleft z+y\circ(x\circ z)-y\circ(x\triangleleft z)-y\triangleleft(x\circ z)+y\triangleleft(x\triangleleft z)\\
&&\ \ \ -(y\circ x)\circ z+(y\circ x)\triangleleft z-x\circ(y\circ z)+x\circ(y\triangleleft z)+x\triangleleft(y\circ z)-x\triangleleft(y\triangleleft z)=0.
\end{eqnarray*}
Hence Eq.~\meqref{eq:defi:L-dend1} holds automatically.

$\meqref{it:equiv2} \Longleftrightarrow \meqref{it:equiv3}$. For $x,y,z\in A$, we have
\begin{eqnarray*}
&&-(\mathcal{L}_{\circ}-\mathcal{L}_{\triangleleft})(x)\mathcal{L}_{\triangleleft}(y)z+\mathcal{L}_{\triangleleft}(y)(\mathcal{L}_{\circ}-\mathcal{L}_{\triangleleft})(x)z+\mathcal{L}_{\triangleleft}(x\circ y)z+\mathcal{L}_{\triangleleft}(y)\mathcal{L}_{\triangleleft}(x)z\\
&&=-x\circ(y\triangleleft z)+x\triangleleft(y\triangleleft z)+y\triangleleft(x\circ z)+(x\circ y)\triangleleft z.
\end{eqnarray*}
Thus Eq.~\meqref{eq:defi:rep of pre-Lie2} holds for the triple $\rbrep{\mathcal{L}_{\circ}-\mathcal{L}_{\triangleleft}}{-\mathcal{L}_{\triangleleft}}{A}$ if and only if Eq.~\meqref{eq:skew-sym op} holds. Furthermore, in this case, we have
\begin{small}
$$(\mathcal{L}_{\circ}-\mathcal{L}_{\triangleleft})(x)(\mathcal{L}_{\circ}-\mathcal{L}_{\triangleleft})(y)z-(\mathcal{L}_{\circ}-\mathcal{L}_{\triangleleft})(x\circ y)z-(\mathcal{L}_{\circ}-\mathcal{L}_{\triangleleft})(y)(\mathcal{L}_{\circ}-\mathcal{L}_{\triangleleft})(x)z+(\mathcal{L}_{\circ}-\mathcal{L}_{\triangleleft})(y\circ x)z$$
\begin{eqnarray*}
&=&x\circ(y\circ z)-x\triangleleft(y\circ z)-x\circ(y\triangleleft z)+x\triangleleft(y\triangleleft z)-(x\circ y)\circ z+(x\circ y)\triangleleft z\\
&& -y\circ(x\circ z)+y\triangleleft(x\circ z)+y\circ(x\triangleleft z)-y\triangleleft(x\triangleleft z)+(y\circ x)\circ z-(y\circ x)\triangleleft z=0.
\end{eqnarray*}
\end{small}
Thus Eq.~\meqref{eq:defi:rep of pre-Lie1} holds for the triple $\rbrep{\mathcal{L}_{\circ}-\mathcal{L}_{\triangleleft}}{-\mathcal{L}_{\triangleleft}}{A}$.

$\meqref{it:equiv3} \Longleftrightarrow \meqref{it:equiv4}$. It follows from Lemma \mref{lem:dual rep}.
\end{proof}

\begin{pro}\mlabel{pro:sp dend cond}
Let $(\mathfrak{g},[-,-],P)$ be a Rota-Baxter Lie algebra of weight zero and $(\mathfrak{g},\circ)$ be the induced pre-Lie algebra.
Let $Q:\frak g\rightarrow \frak g$ be a linear map that is admissible to $(\mathfrak{g},[-,-],P)$.
\begin{enumerate}
    \item \mlabel{11} The triple
   $\rbrep{l_{\mathrm{ad}^{*},Q^{*}}}{r_{\mathrm{ad}^{*},Q^{*}}}{\mathfrak{g}^{*}}$
    is a representation of the pre-Lie algebra $(\mathfrak{g},\circ)$.
    \item\mlabel{22} Define an operation $\triangleleft:\mathfrak{g}\otimes\mathfrak{g}\rightarrow\mathfrak{g}$ on $\frak g$ by
\begin{equation}x\triangleleft y=-Q([x,y]),\;\;\forall x,y\in \frak g.\mlabel{eq:left}\end{equation}
Then $l_{\mathrm{ad}^{*},Q^{*}}=\mathcal{L}^{*}_{\circ}, r_{\mathrm{ad}^{*},Q^{*}}=\mathcal{L}^{*}_{\triangleleft}$.
    \item \mlabel{33} Define an operation $\triangleright:\mathfrak{g}\otimes\mathfrak{g}\rightarrow\mathfrak{g}$ on $\frak g$ by Eq.~\meqref{eq:right}.
Then $(\mathfrak{g},\triangleright,\triangleleft)$ is a compatible \spec \ldend of the pre-Lie algebra $(\frak g, \circ)$.
\end{enumerate}
\end{pro}
\begin{proof}
\meqref{11} Since $Q$ is admissible to $(\mathfrak{g},[-,-],P)$,
$\rbrep{\mathrm{ad}^{*}}{Q^{*}}{\mathfrak{g}^{*}}$
is a representation of the Rota-Baxter Lie algebra $(\mathfrak{g},[-,-],P)$. Hence by Proposition \mref{pro:pre-Lie rep1},
    $\rbrep{l_{\mathrm{ad}^{*},Q^{*}}}{r_{\mathrm{ad}^{*},Q^{*}}}{\mathfrak{g}^{*}}$ is the induced representation of $(\mathfrak{g},\circ)$ from $(\mathrm{ad}^{*},Q^{*};\mathfrak{g}^{*})$.

\meqref{22} For all $x,y\in\mathfrak{g}, a^{*}\in\mathfrak{g}^{*}$, we have
\begin{eqnarray*}
\langle l_{\mathrm{ad}^{*},Q^{*}}(x)a^{*},y\rangle&=&\langle\mathrm{ad}^{*}(P(x))a^{*},y\rangle=-\langle a^*, [P(x),y]\rangle=-\langle a^{*}, x\circ y\rangle=\langle\mathcal{L}^{*}_{\circ}(x)a^{*},y\rangle,\\
\langle r_{\mathrm{ad}^{*},Q^{*}}(x)a^{*},y\rangle&=&-\langle \mathrm{ad}^{*}(x)Q^{*}(a^{*}),y\rangle=\langle Q^*(a^*),[x,y]\rangle=\langle a^{*}, Q([x,y])\rangle=\langle\mathcal{L}^{*} _{\triangleleft}(x)a^{*},y\rangle.
\end{eqnarray*}
Hence $l_{\mathrm{ad}^{*},Q^{*}}=\mathcal{L}^{*}_{\circ}, r_{\mathrm{ad}^{*},Q^{*}}=\mathcal{L}^{*}_{\triangleleft}$.

\meqref{33} By Items (\mref{11}) and (\mref{22}),
$\rbrep{\mathcal{L}^{*}_{\circ}}{\mathcal{L}^{*}_{\triangleleft}}{\mathfrak{g}^{*}}$
is a representation of $(\mathfrak{g},\circ)$. Thus $(\mathfrak{g},\triangleright,\triangleleft)$ is a compatible \spec \ldend of $(\frak g,\circ)$ by Proposition \mref{pro:equivalence}.
\end{proof}

\begin{cor}\mlabel{cor:from RB Lie alg to sp L-dendri}
Let $(\mathfrak{g},[-,-],P)$ be a Rota-Baxter Lie algebra of weight zero and $(\mathfrak{g},\circ)$ be the induced pre-Lie algebra.
\begin{enumerate}
\item \mlabel{it:11}There is a compatible \spec \ldend $(\frak g, \triangleright,\triangleleft)$ of the pre-Lie algebra $(\mathfrak{g},\circ)$ defined by
\begin{equation}
\mlabel{eq:con} x\triangleleft y=P([x,y]), \quad x\triangleright y=x\circ y-x\triangleleft y=[P(x),y]-P([x,y]),\;\;\forall x,y\in \frak g.
\end{equation}
\item \mlabel{it:12} Suppose that $\mathcal{B}$ is a \nonsy invariant
 bilinear form on $(\mathfrak{g},[-,-])$. Let $\widehat{P}:\mathfrak{g}\rightarrow \mathfrak{g}$ be the adjoint linear operator of $P$ with respect to $\mathcal{B}$ as defined in Eq.~\meqref{eq:adop}.
Then there is a compatible \spec \ldend $(\mathfrak{g},\triangleright,\triangleleft)$ of the pre-Lie algebra $(\frak g,\circ)$ defined by
\begin{equation}\mlabel{eq:pro:from RB Lie alg to sp L-dendri3}
x\triangleleft y=-\widehat{P}([x,y]), \quad x\triangleright y=x\circ y-x\triangleleft y=[P(x),y]+\widehat P([x,y]),\quad \forall x,y\in \mathfrak{g}.
\end{equation}
\end{enumerate}
\end{cor}
\begin{proof}
(\mref{it:11}). By Corollary~\mref{cor:adm}, $-P$ is admissible to $(\mathfrak{g},[-,-],P)$. Hence the conclusion follows from Proposition~\mref{pro:sp dend cond}.

(\mref{it:12}). By Proposition~\mref{pro:invariance}, $\widehat P$ is admissible to $(\mathfrak{g},[-,-],P)$. Then the conclusion follows from Proposition~\mref{pro:sp dend cond}.
    \end{proof}

\begin{rmk}
By~\mcite{BLN}, for a Rota-Baxter Lie algebra
$(\mathfrak{g},[-,-],P)$ of weight zero and the induced pre-Lie
algebra $(\frak g,\circ)$, the operator $P$ is also a Rota-Baxter operator of weight zero on the pre-Lie algebra $(\frak g, \circ)$. Further, for the binary operations
$$x\triangleright' y=P(x)\circ y=[P^2(x), y],\;\;x\triangleleft'y=x\circ P(y)=[P(x),P(y)],\;\;\forall x,y\in \frak g,$$
the triple
$(\frak g, \triangleright',\triangleleft')$ is an \ldend.
Obviously, this \ldend is special. In general, it is different from the \spec \ldend defined by Eq.~\meqref{eq:con} and is not compatible with the pre-Lie algebra $(\frak g,\circ)$.
\end{rmk}

We recall another lemma before the next application.

\begin{lem}\mlabel{lem:from left invariance to special} \mcite{BHC}
Let $(A,\circ)$ be a pre-Lie algebra with a \nonsy left-invariant bilinear form $\mathcal{B}$. Then there is a compatible \spec \ldend $(A,\triangleright,\triangleleft)$ with the multiplications $\triangleright$ and $\triangleleft$ defined by
    \begin{equation}\mlabel{eq:lem:from left invariance to special1}
        \mathcal{B}(x\triangleleft y,z)=\mathcal{B}(x,z\circ y),\quad  x\triangleright y=x\circ y-x\triangleleft y,\;\;\forall x,y,z\in A.
    \end{equation}
\end{lem}

\begin{cor}
Let $(\mathfrak{g},[-,-],P)$ be a Rota-Baxter Lie algebra of weight zero and $(\mathfrak{g},\circ)$ be the induced pre-Lie algebra.
Suppose that $\mathcal{B}$ is a \nonsy invariant
 bilinear form on $(\mathfrak{g},[-,-])$. Let $\widehat{P}:\mathfrak{g}\rightarrow \mathfrak{g}$ be the adjoint linear operator of $P$ with respect to $\mathcal{B}$.
 Then the \spec \ldend defined by
 Eq.~\meqref{eq:pro:from RB Lie alg to sp L-dendri3} coincides with the one defined by Eq.~\meqref{eq:lem:from left invariance to special1}.
\end{cor}

\begin{proof}  On the one hand,  by Corollary \mref{cor:from RB Lie alg to sp L-dendri} (\mref{it:12}), there is a compatible \spec \ldend $(\mathfrak{g},\triangleright_{1},\triangleleft_{1})$ of the pre-Lie algebra $(\frak g, \circ)$ defined by Eq.~\meqref{eq:pro:from RB Lie alg to sp L-dendri3}.
On the other hand, by Proposition \mref{pro:from invariance to left invariance}, $\mathcal{B}$ is left-invariant on the pre-Lie algebra $(\mathfrak{g},\circ)$.
Hence by Lemma \mref{lem:from left invariance to special}, there is another compatible \spec \ldend $(\mathfrak{g},\triangleright_{2},\triangleleft_{2})$ of the pre-Lie algebra $(\frak g,\circ)$ defined by Eq.~\meqref{eq:lem:from left invariance to special1}. For all $x,y,z\in\mathfrak{g}$, we have
$$\mathcal{B}(x\triangleleft_{2}y,z)=\mathcal{B}(x,z\circ y)=\mathcal{B}(x,[P(z),y])=-\mathcal{B}([x,y],P(z))=-\mathcal{B}(\widehat{P}([x,y]),z)=\mathcal{B}(x\triangleleft_{1}y,z).$$
Thus $x\triangleleft_{2}y=x\triangleleft_{1}y$.  Moreover
$$x\triangleright_{2}y=x\circ y-x\triangleleft_2y=x\circ y-x\triangleleft_1y  =x\triangleright_{1}y,\;\;\forall x,y\in \frak g.$$ Hence the two \spec \ldends $(\mathfrak{g},\triangleright_{1},\triangleleft_{1})$ and $(\mathfrak{g},\triangleright_{2},\triangleleft_{2})$ coincide.
\end{proof}

\begin{ex}\mlabel{ex:ex}
Let $(\frak g,[-,-])$ be the 3-dimensional simple Lie algebra
$\frak s\frak l(2,\mathbb{K})$ with a basis $\left\{
 x=\left(\begin{matrix} 0&1\cr
    0&0\cr\end{matrix}\right),h=\left(\begin{matrix} 1&0\cr
    0&-1\cr\end{matrix}\right),y=\left(\begin{matrix} 0&0\cr
    1&0\cr\end{matrix}\right)\right\}$ and with the multiplication
\begin{equation*}\mlabel{eq:E1}
    [h,x]=2x,[h,y]=-2y,[x,y]=h.
\end{equation*}
Define a linear operator $P:\mathfrak{g}\rightarrow\mathfrak{g}$ by
\begin{equation*}\mlabel{eq:E2}
P(x)=x+y, P(h)=2h+4y, P(y)=x-2h-3y.
\end{equation*}
Then $P$ is a Rota-Baxter operator of weight zero on $(\frak g,[-,-])$, and the induced pre-Lie algebra $(\mathfrak{g},\circ)$ from Eq.~\meqref{eq:pro:from invariance to left invariance1} is given by
\begin{eqnarray*}
&& x\circ x=-h, x\circ h=-2x+2y, x\circ y=h,\\
&& h\circ x=4x-4h, h\circ h=8y, h\circ y=-4y,\\
&& y\circ x=3h-4y, y\circ y=h+4y, y\circ h=-2x-6y.
\end{eqnarray*}
Moreover, there is a \nonsy invariant bilinear form $\mathcal{B}$ on $(\mathfrak{g},[-,-])$ whose nonzero values are 
\vspace{-.2cm}
$$\mathcal{B}(x,y)=\mathcal{B}(y,x)=1, \quad \mathcal{B}(h,h)=2.$$
    The adjoint linear operator $\widehat{P}$ of $P$ with respect to $\mathcal{B}$ is given by
\begin{equation*}\mlabel{eq:E4}
    \widehat{P}(x)=-3x+2h+y, \widehat{P}(h)=-4x+2h, \widehat{P}(y)=x+y.
\end{equation*}
Thus a  compatible \spec \ldend
$(\mathfrak{g},\triangleright,\triangleleft)$
  of the pre-Lie algebra $(\frak g,\circ)$ is given by
$$
x\triangleleft h=-\widehat{P}([x,h])=-6x+4h+2y, x\triangleleft y=-\widehat{P}([x,y])=4x-2h, h\triangleleft y=-\widehat{P}([h,y])=2x+2y,
$$
and $a\triangleright b=a\circ b-a\triangleleft b$ for all $a,b\in \frak g$. Also note
that $\widehat{P}$ commutes with $P$.
\end{ex}

\section{Rota-Baxter Lie bialgebras and \spec \ldendbs}\mlabel{sec:3}

In this section we introduce the notion of Rota-Baxter Lie bialgebras which comes naturally from the notions of Manin triples of Rota-Baxter Lie
algebras and matched pairs of Rota-Baxter Lie algebras. We establish the explicit relationship between
Rota-Baxter Lie bialgebras of weight zero and \spec \ldendbs introduced in \mcite{BHC}. Similar relations are also established from the Manin triples (resp. the matched pairs) of Rota-Baxter Lie algebras to those of \spec \ldends.

\subsection{Rota-Baxter Lie bialgebras}\mlabel{sec:3.1}\
Following the Manin triple approach to Lie bialgebras~\mcite{Dri}, we derive the notion of Rota-Baxter Lie bialgebras from Manin triples and matched pairs of Rota-Baxter Lie algebras.

\subsubsection{Manin triples of Rota-Baxter Lie algebras}\

Recall \mcite{CP}  that a (standard) \textbf{Manin triple of Lie algebras} is a triple $(\mathfrak{g}\oplus \mathfrak{g}^{*},\mathfrak{g},\mathfrak{g}^{*})$ of Lie algebras, such that  $(\mathfrak{g},[-,-]_{\mathfrak{g}})$ and $(\mathfrak{g}^{*},[-,-]_{\mathfrak{g}^{*}})$ are Lie subalgebras of the Lie algebra $(\mathfrak{g}\oplus \mathfrak{g}^{*},[-,-])$, and the natural \nonsy bilinear form $\mathcal{B}_{d}$ on $(\mathfrak{g}\oplus \mathfrak{g}^{*},[-,-])$ given by
\begin{equation}\mlabel{eq:natrual bilinear form}
\mathcal{B}_{d}(x+a^{*},y+b^{*})=\langle x,b^{*}\rangle+\langle a^{*},y\rangle,\;\;\forall x,y\in\mathfrak{g}, a^{*}, b^{*}\in\mathfrak{g}^{*}
\end{equation}
is invariant. We extend this notion to Rota-Baxter Lie algebras.

\begin{defi}\mlabel{defi:MT}
A {\bf Manin triple of Rota-Baxter Lie algebras} is a triple
$((\mathfrak{g}\oplus\mathfrak{g}^{*},P_{\frakg\oplus \frakg^*}),$
$(\mathfrak{g},P),(\mathfrak{g}^{*}$, $Q^{*}))$ of Rota-Baxter Lie
algebras such that $(\mathfrak{g}\oplus
\mathfrak{g}^{*},\mathfrak{g},\mathfrak{g}^{*})$ is a Manin triple
of Lie algebras and $P_{\frakg\oplus \frakg^*}=P+Q^{*}$. Then we
denote the Manin triple by
$((\mathfrak{g}\oplus\mathfrak{g}^{*},P+Q^{*}),(\mathfrak{g},P),
(\mathfrak{g}^{*},Q^{*}))$.
\end{defi}

By definition, for a Manin triple of Rota-Baxter Lie algebras
$((\mathfrak{g}\oplus\mathfrak{g}^{*},P+Q^{*}),(\mathfrak{g},P),
(\mathfrak{g}^{*}$, $Q^{*}))$, the triples
$(\mathfrak{g},[-,-]_{\mathfrak{g}},P)$ and
$(\mathfrak{g}^{*},[-,-]_{\mathfrak{g}^{*}},Q^{*})$ are clearly
Rota-Baxter Lie subalgebras of
$(\mathfrak{g}\oplus\mathfrak{g}^{*},[-,-],P+Q^{*})$. Moreover, we
have the following conclusion.

\begin{lem}\mlabel{lem:MP adm}
Let $((\mathfrak{g}\oplus\mathfrak{g}^{*},P+Q^{*}),(\mathfrak{g},P),(\mathfrak{g}^{*},Q^{*}))$ be a Manin triple of Rota-Baxter Lie algebras.
\begin{enumerate}
\item The adjoint $\widehat{P+Q^{*}}$ of $P+Q^{*}$ with respect to $\mathcal{B}_{d}$ is $Q+P^{*}$. Further $Q+P^{*}$ is admissible to $(\mathfrak{g}\oplus \mathfrak{g}^{*},[-,-],P+Q^{*})$.

\item $Q$ is admissible to $(\mathfrak{g},[-,-]_{\mathfrak{g}},P)$.

\item $P^{*}$ is admissible to $(\mathfrak{g}^{*},[-,-]_{\mathfrak{g}^{*}},Q^{*})$.
\end{enumerate}
\end{lem}
\begin{proof}
The proof is similar to the one of \cite[Lemma 3.11]{BGM}.
\end{proof}

Then by Corollary~\mref{cor:from RB Lie alg to sp L-dendri} \meqref{it:12}, we obtain
\begin{cor}\mlabel{cor:MP adm}
Let $((\mathfrak{g}\oplus\mathfrak{g}^{*},P+Q^{*}),(\mathfrak{g},P),(\mathfrak{g}^{*},Q^{*}))$ be a Manin triple of Rota-Baxter Lie algebras of weight zero.
There is a \spec \ldend $(\mathfrak{g}\oplus\mathfrak{g}^{*},\triangleright,\triangleleft)$ defined by
\begin{equation}\mlabel{eq:cor:MP adm7}
(x+a^{*})\triangleleft(y+b^{*})=-(Q+P^{*})([x+a^{*},y+b^{*}]),
\end{equation}
\begin{equation}\mlabel{eq:cor:MP adm8}
(x+a^{*})\triangleright(y+b^{*})=(x+a^{*})\circ(y+b^{*})-(x+a^{*})\triangleleft(y+b^{*}),\;\;\forall x,y\in\mathfrak{g}, a^{*},b^{*}\in\mathfrak{g}^{*},
\end{equation}
which contains $(\mathfrak{g},\triangleright_{\mathfrak{g}},\triangleleft_{\mathfrak{g}})$ and $(\mathfrak{g}^{*},\triangleright_{\mathfrak{g}^{*}},\triangleleft_{\mathfrak{g}^{*}})$ as special L-dendriform subalgebras, where
\begin{equation}\mlabel{eq:cor:MP adm9}
x\triangleleft_{\mathfrak{g}}y=-Q([x,y]_{\mathfrak{g}}),\quad  x\triangleright_{\mathfrak{g}} y=x\circ_{\mathfrak{g}}y-x\triangleleft_{\mathfrak{g}}y,\;\;\forall x,y\in\mathfrak{g},
\end{equation}
\begin{equation}\mlabel{eq:cor:MP adm10}
a^{*}\triangleleft_{\mathfrak{g}^{*}}b^{*}=-P^{*}([a^{*}, b^{*}]_{\mathfrak{g}^{*}}), \quad  a^{*}\triangleright_{\mathfrak{g}^{*}}b^{*}=a^{*}\circ_{\mathfrak{g}^{*}}b^{*}-a^{*}\triangleleft_{\mathfrak{g}^{*}}b^{*},\;\;\forall a^{*},b^{*}\in\mathfrak{g}^{*}.
\end{equation}
\end{cor}

\subsubsection{Matched pairs of Rota-Baxter Lie algebras}\

Let $(\mathfrak{g},[-,-]_{\mathfrak{g}})$ and $(\mathfrak{h},[-,-]_{\mathfrak{h}})$ be Lie
algebras, with their respective representations
$\mrep{\rho_{\mathfrak{g}}}{\mathfrak{h}}$ and $\mrep{\rho_{\mathfrak{h}}}{\mathfrak{g}}$.
Then $(\mathfrak{g},\mathfrak{h},\rho_{\mathfrak{g}},\rho_{\mathfrak{h}})$ is called a \textbf{matched pair of Lie algebras} \mcite{Maj} 
if
\begin{eqnarray*}
&\rho_{\mathfrak{g}}(x)[a,b]_{\mathfrak{h}}-[\rho_{\mathfrak{g}}(x)a,b]_{\mathfrak{h}}-[a,\rho_{\mathfrak{g}}(x)b]_{\mathfrak{h}}+\rho_{\mathfrak{g}}(\rho_{\mathfrak{h}}(a)x)b-\rho_{\mathfrak{g}}(\rho_{\mathfrak{h}}(b)x)a=0,&\\
&\rho_{\mathfrak{h}}(a)[x,y]_{\mathfrak{g}}\!-[\rho_{\mathfrak{h}}(a)x,y]_{\mathfrak{g}}\!-[x,\rho_{\mathfrak{h}}(a)y]_{\mathfrak{g}}\!+\rho_{\mathfrak{h}}(\rho_{\mathfrak{g}}(x)a)y-\rho_{\mathfrak{h}}(\rho_{\mathfrak{g}}(y)a)x=0,  \forall x,y\in\frakg, a,b\in \frakh.&
\end{eqnarray*}
For Lie algebras $(\mathfrak{g},[-,-]_{\mathfrak{g}})$ and $(\mathfrak{h},[-,-]_{\mathfrak{h}})$ with linear maps $\rho_{\mathfrak{g}}:\mathfrak{g}\rightarrow\mathrm{End}(\mathfrak{h}), \rho_{\mathfrak{h}}:\mathfrak{h}\rightarrow\mathrm{End}(\mathfrak{g})$,
 there is a Lie algebra structure
on the vector space $\mathfrak{g}\oplus \mathfrak{h}$ given by
\begin{equation}\mlabel{eq:Lie}
[x+a,y+b]=[x,y]_{\mathfrak{g}}+\rho_{\mathfrak{h}}(a)y-\rho_{\mathfrak{h}}(b)x+[a,b]_{\mathfrak{h}}+\rho_{\mathfrak{g}}(x)b-\rho_{\mathfrak{g}}(y)a, \forall x,y\in \mathfrak{g}, a,b\in \mathfrak{h}
\end{equation}
if and only if $(\mathfrak{g},\mathfrak{h},\rho_{\mathfrak{g}},\rho_{\mathfrak{h}})$ is a matched pair of Lie algebras.
 We denote the resulting Lie algebra $(\mathfrak{g}\oplus\mathfrak{h},[-,-])$ by $\mathfrak{g}\bowtie^{\rho_{\mathfrak{g}}}_{\rho_{\mathfrak{h}}}\mathfrak{h}$ or simply $\mathfrak{g}\bowtie\mathfrak{h}$.

We extend this construction to Rota-Baxter Lie algebras.
\begin{defi}\mlabel{defi:MP RB Lie}
A \textbf{matched pair of Rota-Baxter Lie algebras} is a quadruple
$((\mathfrak{g},P_{\mathfrak{g}}),$
$(\mathfrak{h},P_{\mathfrak{h}})$, $\rho_{\mathfrak{g}}$,
$\rho_{\mathfrak{h}})$, such that
$(\mathfrak{g},[-,-]_{\mathfrak{g}},P_{\mathfrak{g}})$ and
$(\mathfrak{h},[-,-]_{\mathfrak{h}},P_{\mathfrak{h}})$ are
Rota-Baxter Lie algebras,
$\rbrep{\rho_{\mathfrak{g}}}{P_{\mathfrak{h}}}{\mathfrak{h}}$ is a
representation of $(\mathfrak{g},[-,-]_{\mathfrak{g}}$,
$P_{\mathfrak{g}})$,
$\rbrep{\rho_{\mathfrak{h}}}{P_{\mathfrak{g}}}{\mathfrak{g}}$
 is a representation of $(\mathfrak{h},[-,-]_{\mathfrak{h}},P_{\mathfrak{h}})$, and $(\mathfrak{g},\mathfrak{h},\rho_{\mathfrak{g}},\rho_{\mathfrak{h}})$ is a matched pair of Lie algebras.
\end{defi}

Adapting the argument for Theorem~3.4 in~\cite{BGM} from associative algebras to Lie algebras, we obtain
\begin{pro}\mlabel{pro:MP RB}
Let $(\mathfrak{g},[-,-]_{\mathfrak{g}},P_{\mathfrak{g}})$ and $(\mathfrak{h},[-,-]_{\mathfrak{h}},P_{\mathfrak{h}})$ be Rota-Baxter Lie algebras and $(\mathfrak{g},\mathfrak{h},\rho_{\mathfrak{g}},\rho_{\mathfrak{h}})$ be a matched pair of Lie algebras. With the resulting Lie algebra $\mathfrak{g}\bowtie\mathfrak{h}$ given by Eq.~\meqref{eq:Lie}, $(\mathfrak{g}\bowtie\mathfrak{h},[-,-],P_{\frak g}+P_{\frak h})$ is a Rota-Baxter Lie algebra if and only if  $((\mathfrak{g},P_{\mathfrak{g}}),(\mathfrak{h},P_{\mathfrak{h}}),\rho_{\mathfrak{g}}$,
$\rho_{\mathfrak{h}})$ is a matched pair of the Rota-Baxter Lie algebras $(\mathfrak{g},[-,-]_{\mathfrak{g}},P_{\mathfrak{g}})$ and $(\mathfrak{h},[-,-]_{\mathfrak{h}},P_{\mathfrak{h}})$.
\end{pro}

For Lie algebras $(\mathfrak{g},[-,-]_{\mathfrak{g}})$ and $(\mathfrak{g}^{*},[-,-]_{\mathfrak{g}^{*}})$, there is a Manin triple of Lie algebras $(\mathfrak{g}\oplus\mathfrak{g}^{*},\mathfrak{g},\mathfrak{g}^{*})$ if and only if $(\mathfrak{g},\mathfrak{g}^{*},\mathrm{ad}^{*}_{\mathfrak{g}},\mathrm{ad}^{*}_{\mathfrak{g}^{*}})$ is a matched pair of Lie algebras~\mcite{CP}.
This characterization can be extended to Rota-Baxter Lie algebras by the same argument as~\cite[Theorem 3.12]{BGM} for Rota-Baxter associative algebras.

\begin{pro}\mlabel{pro:12 RB Lie}
Let $(\mathfrak{g},[-,-]_{\mathfrak{g}},P)$ be a Rota-Baxter Lie algebra. Suppose that there is a Rota-Baxter Lie algebra $(\mathfrak{g}^{*},[-,-]_{\mathfrak{g}^{*}},Q^{*})$ on the dual space $\mathfrak{g}^{*}$. Then there is a Manin triple of Rota-Baxter Lie algebras $((\mathfrak{g}\oplus\mathfrak{g}^{*},P+Q^{*}),(\mathfrak{g},P),(\mathfrak{g}^{*},Q^{*}))$ if and only if $((\mathfrak{g},P),(\mathfrak{g}^{*},Q^{*}),\mathrm{ad}^{*}_{\mathfrak{g}},\mathrm{ad}^{*}_{\mathfrak{g}^{*}})$ is a matched pair of Rota-Baxter Lie algebras.
\end{pro}

\subsubsection{Rota-Baxter Lie bialgebras}\

A {\bf Lie coalgebra} is a pair $(\frak g,\delta)$ with a vector space $\frak g$ and a linear map $\delta:\frak
g\rightarrow \frak g\otimes \frak g$ such that
\begin{enumerate}
\item $\delta$ is {\bf
co-antisymmetric}, that is, $\delta=-\tau \delta$ for
the flip map $\tau:\frak g\otimes \frak g\rightarrow \frak
g\otimes \frak g$, and
\item the {\bf co-Jacobian identity} holds:
\begin{equation}
({\rm id} +\sigma+\sigma^2)({\rm id} \otimes \delta)\delta =0,
\end{equation}
where $\sigma(x\otimes y\otimes z):=z\otimes x\otimes y$ for $x,
y, z\in \frak g$.
\end{enumerate}

\begin{defi}\mlabel{defi:Lie bialg}\mcite{CP}
A \textbf{Lie bialgebra} is a triple $(\mathfrak{g},[-,-],\delta)$ consisting of a vector space $\mathfrak{g}$ and linear maps $[-,-]:\mathfrak{g}\otimes\mathfrak{g}\rightarrow\mathfrak{g}$, $\delta:\mathfrak{g}\rightarrow\mathfrak{g}\otimes\mathfrak{g}$ such that
\begin{enumerate}
\item $(\mathfrak{g},[-,-])$ is a Lie algebra.
\item $(\mathfrak{g},\delta)$ is a Lie coalgebra.
\item $\delta$ is a 1-cocycle of $\mathfrak{g}$ with values in $\mathfrak{g}\otimes \mathfrak{g}$, that is,
\begin{equation}\mlabel{eq:defi:Lie bialg1}
\delta([x,y])=(\mathrm{ad}(x)\otimes \mathrm{id}+\mathrm{id}\otimes\mathrm{ad}(x))\delta(y)-(\mathrm{ad}(y)\otimes \mathrm{id}+\mathrm{id}\otimes\mathrm{ad}(y))\delta(x),\forall x,y\in \mathfrak{g}.
\end{equation}
\end{enumerate}
\end{defi}

We now extend the notion of Lie bialgebras to Rota-Baxter Lie bialgebras.

\begin{defi}\mlabel{defi:RB Lie co}
A \textbf{Rota-Baxter Lie coalgebra  (of weight $\lambda$)} is a triple $(\mathfrak{g},\delta,Q)$, where $(\mathfrak{g},\delta)$ is a Lie coalgebra, and $Q:\mathfrak{g}\rightarrow\mathfrak{g}$ is a linear map such that
\begin{equation}\mlabel{eq:defi:RB Lie co1}
(Q\otimes Q)\delta(x)=(Q\otimes\mathrm{id}+\mathrm{id}\otimes Q)\delta(Q(x))+\lambda\delta(Q(x)),\;\;\forall x\in\mathfrak{g}.
\end{equation}
\end{defi}

Generalizing the well-known duality between Lie coalgebras and Lie algebras, for a finite dimensional vector space $\mathfrak{g}$,
the pair $(\mathfrak{g}^{*},[-,-]_{\mathfrak{g}^{*}},Q^{*})$ is a Rota-Baxter Lie algebra if and only if $(\mathfrak{g},\delta,Q)$ is a Rota-Baxter Lie coalgebra, where $\delta:\mathfrak{g}\rightarrow\mathfrak{g}\otimes\mathfrak{g}$ is the linear dual of $[-,-]_{\mathfrak{g}^{*}}:\mathfrak{g}^{*}\otimes\mathfrak{g}^{*}\rightarrow\mathfrak{g}^{*}$, that is,
\begin{equation}\mlabel{eq:linear dual}
    \langle\delta(x), a^{*}\otimes b^{*}\rangle=\langle x,[a^{*},b^{*}]_{\mathfrak{g}^{*}}\rangle,\;\;\forall x\in\mathfrak{g}, a^{*},b^{*}\in\mathfrak{g}^{*}.
\end{equation}
In this case, the Lie algebra structure on $\mathfrak{g}^{*}$ is also denoted by $\delta^{*}$, that is,
$$\delta^{*}(a^{*}\otimes b^{*})=[a^{*},b^{*}]_{\mathfrak{g}^{*}},\;\;\forall a^*, b^* \in \frak g^*.$$
Moreover, for a linear map $P:\mathfrak{g}\rightarrow \mathfrak{g}$, the condition that $P^{*}$ is admissible to the Rota-Baxter Lie algebra $(\mathfrak{g}^{*},[-,-]_{\mathfrak{g}^{*}},Q^{*})$, that is, for all $a^{*},b^{*}\in\mathfrak{g}^{*}$,
\begin{equation}\mlabel{adm cond 2}
P^{*}([Q^{*}(a^{*}),b^{*}]_{\mathfrak{g}^{*}})-[Q^{*}(a^{*}),P^{*}(b^{*})]_{\mathfrak{g}^{*}}-P^{*}([a^{*},P^{*}(b^{*})]_{\mathfrak{g}^{*}})-\lambda[a^{*},P^{*}(b^{*})]_{\mathfrak{g}^{*}}=0, 
\end{equation}
can be rewritten in terms of the comultiplication $\delta$ as follows:
\begin{equation}\mlabel{eq:defi:RB Lie bialg1}
    (P\otimes Q)\delta(x)+(P\otimes\mathrm{id}-\mathrm{id}\otimes Q)\delta(P(x))+\lambda(P\otimes\mathrm{id})\delta(x)=0,\;\;\forall x\in\mathfrak{g}.
\end{equation}

\begin{defi}\mlabel{defi:RB Lie bialg}
A \textbf{Rota-Baxter Lie bialgebra  (of weight $\lambda$)} is
a vector space $\mathfrak{g}$ together with linear maps
$$[-,-]:\mathfrak{g}\otimes\mathfrak{g}\rightarrow\mathfrak{g},\ \  \delta:\mathfrak{g}\rightarrow\mathfrak{g}\otimes\mathfrak{g},\ \  P,Q:\mathfrak{g}\rightarrow\mathfrak{g}$$
such that
\begin{enumerate}
\item the triple $(\mathfrak{g},[-,-],\delta)$ is a Lie bialgebra.
\item the triple $(\mathfrak{g},[-,-],P)$ is a Rota-Baxter Lie algebra.
\item the triple $(\mathfrak{g},\delta,Q)$ is a Rota-Baxter Lie coalgebra.
\item $Q$ is admissible to $(\mathfrak{g},[-,-],P)$, that is, Eq.~\meqref{eq:cor:adm1} holds.
\item $P^{*}$ is admissible to $(\mathfrak{g}^{*},\delta^{*},Q^{*})$,  that is, Eq.~\meqref{eq:defi:RB Lie bialg1} holds.
\end{enumerate}
We denote the Rota-Baxter Lie bialgebra by $(\mathfrak{g},[-,-],P,\delta,Q)$ or simply $(\mathfrak{g},P,\delta,Q)$.
\end{defi}

Let $(\mathfrak{g},[-,-]_{\mathfrak{g}})$ be a Lie algebra. Suppose that there is a Lie algebra $(\mathfrak{g}^{*},[-,-]_{\mathfrak{g}^{*}})$ on the dual space $\mathfrak{g}^{*}$. Let $\delta:\mathfrak{g}\rightarrow\mathfrak{g}\otimes\mathfrak{g}$ be the linear dual of $[-,-]_{\mathfrak{g}^{*}}:\mathfrak{g}^{*}\otimes \mathfrak{g}^{*}\rightarrow\mathfrak{g}^{*}$.
Then  $(\mathfrak{g},[-,-]_{\mathfrak{g}},\delta)$ is a Lie bialgebra if and only if $(\mathfrak{g},\mathfrak{g}^{*},\mathrm{ad}^{*}_{\mathfrak{g}},\mathrm{ad}^{*}_{\mathfrak{g}^{*}})$ is a matched pair of Lie algebras \mcite{CP}. Generalizing this fact to the context of Rota-Baxter Lie algebras, we have

\begin{pro}\mlabel{pro:23 RB Lie}
Let $(\mathfrak{g},[-,-]_{\mathfrak{g}},P)$ be a Rota-Baxter Lie algebra. Suppose that there is a Rota-Baxter Lie algebra $(\mathfrak{g}^{*},[-,-]_{\mathfrak{g}^{*}},Q^{*})$ on the dual space $\mathfrak{g}^{*}$. Let $\delta:\mathfrak{g}\rightarrow\mathfrak{g}\otimes\mathfrak{g}$ be the linear dual of $[-,-]_{\mathfrak{g}^{*}}:\mathfrak{g}^{*}\otimes \mathfrak{g}^{*}\rightarrow\mathfrak{g}^{*}$. Then $(\mathfrak{g},[-,-]_{\mathfrak{g}},P,\delta,Q)$ is a Rota-Baxter Lie bialgebra if and only if $((\mathfrak{g},P),(\mathfrak{g}^{*},Q^{*}),\mathrm{ad}^{*}_{\mathfrak{g}},\mathrm{ad}^{*}_{\mathfrak{g}^{*}})$ is a matched pair of Rota-Baxter Lie algebras.
\end{pro}
\begin{proof}
The proof follows the one of~\cite[Theorem 3.5]{BGM}.
\end{proof}

Combining Theorems \mref{pro:12 RB Lie} and \mref{pro:23 RB Lie}, we obtain
\begin{thm}\mlabel{thm:123 RB Lie}
Let $(\mathfrak{g},[-,-]_{\mathfrak{g}},P)$ be a Rota-Baxter Lie algebra. Suppose that there is a Rota-Baxter Lie algebra $(\mathfrak{g}^{*},[-,-]_{\mathfrak{g}^{*}},Q^{*})$ on the dual space $\mathfrak{g}^{*}$. Let $\delta:\mathfrak{g}\rightarrow\mathfrak{g}\otimes\mathfrak{g}$ be the linear dual of $[-,-]_{\mathfrak{g}^{*}}:\mathfrak{g}^{*}\otimes \mathfrak{g}^{*}\rightarrow\mathfrak{g}^{*}$. Then the following conditions are equivalent.
\begin{enumerate}
\item There is a Manin triple $((\mathfrak{g}\oplus\mathfrak{g}^{*},P+Q^{*}),(\mathfrak{g},P),(\mathfrak{g}^{*},Q^{*}))$ of Rota-Baxter Lie algebras.
\item $((\mathfrak{g},P),(\mathfrak{g}^{*},Q^{*}),\mathrm{ad}^{*}_{\mathfrak{g}},\mathrm{ad}^{*}_{\mathfrak{g}^{*}})$ is a matched pair of Rota-Baxter Lie algebras.
\item $(\mathfrak{g},[-,-]_{\mathfrak{g}},P,\delta,Q)$ is a Rota-Baxter Lie bialgebra.
\end{enumerate}
\end{thm}

\subsection{Rota-Baxter Lie bialgebras of weight zero
and \spec \ldendbs}\mlabel{sec:3.2}
Here we show that the well-known connection between Rota-Baxter Lie algebras of weight zero and  pre-Lie algebras has its bialgebra enrichment as a connection between Rota-Baxter Lie bialgebras of weight zero and \spec \ldendbs. Furthermore, as in the case of Lie bialgebras, this connection among the various bialgebra structures is also characterized by suitable Manin triples and matched pairs.

\subsubsection{Manin triples of pre-Lie algebras with respect to the \nonsy left-invariant bilinear form $\mathcal{B}_{d}$
}\
\begin{defi}\mlabel{defi:MT LSA} \mcite{BHC}
Suppose that there are three pre-Lie algebras $(A,\circ_{A})$, $(A^{*},\circ_{A^{*}})$ and $(A\oplus A^{*},\circ)$ such that $(A,\circ_{A})$ and $(A^{*},\circ_{A^{*}})$ are pre-Lie subalgebras of $(A\oplus A^{*},\circ)$, and the natural  \nonsy bilinear form $\mathcal{B}_{d}$ on $A\oplus A^{*}$ defined by Eq.~\meqref{eq:natrual bilinear form}
is left-invariant. Then the triple $((A\oplus A^{*},\circ), (A,\circ_A), (A^{*},\circ_{A^*}))$ is called a \textbf{Manin  triple of pre-Lie algebras} with respect to the \nonsy left-invariant bilinear form $\mathcal{B}_{d}$ associated to $(A,\circ_{A})$ and $(A^{*},\circ_{A^{*}})$,
and is denoted by $(A\bowtie A^*, A,A^*)$.
\end{defi}

\begin{cor}\mlabel{cor:from invariance to left invariance}
Let $((\mathfrak{g}\oplus\mathfrak{g}^{*},P+Q^{*}),(\mathfrak{g},P),(\mathfrak{g}^{*}, Q^{*}))$ be a Manin triple of Rota-Baxter Lie algebras of weight zero. Then the induced pre-Lie algebras $($defined by Eq.~\meqref{eq:pro:from invariance to left invariance1}$)$ from the three Rota-Baxter Lie algebras form a Manin triple $((\frak g\bowtie \frak g^*,\circ), (\frak g,\circ_\frakg),(\frak g^*,\circ_{\frakg^*}))$ of pre-Lie algebras with respect to the \nonsy left-invariant bilinear form $\mathcal{B}_{d}$ associated to $(\frak g,\circ_{\frak g})$ and $(\frak g^{*},\circ_{\frak g^{*}})$.
\end{cor}

\begin{proof}
It follows from the definition that the induced pre-Lie algebra $(\frakg\oplus \frakg^*,\circ)$ from the Rota-Baxter Lie algebra $(\frakg\oplus\frakg^*,P+Q^*)$ contains the induced pre-Lie algebras $(\frakg,\circ_\frakg)$ and $(\frakg^*,\circ_{\frakg^*})$ as pre-Lie subalgebras. The left-invariance of ${\mathcal B}_d$ follows from Proposition \mref{pro:from invariance to left invariance}.
\end{proof}

\subsubsection{Matched pairs of pre-Lie algebras}\
\begin{defi}\mlabel{defi:MP pre-Lie} \mcite{Bai2008}
Let $(A,\circ_{A})$ and $(B,\circ_{B})$ be pre-Lie algebras and $(\mathfrak{g}(A),[-,-]_{A})$ and  $(\mathfrak{g}(B),[-,-]_{B})$ be their sub-adjacent Lie algebras respectively. Suppose that there are linear maps $l_{A},r_{A}:A\rightarrow\mathrm{End}(B)$ and $l_{B},r_{B}:B\rightarrow\mathrm{End}(A)$ such that
$\rbrep{l_{A}}{r_{A}}{B}$
is a representation of $(A,\circ_{A})$,
$\rbrep{l_{B}}{r_{B}}{A}$
is a representation of $(B,\circ_{B})$, and for all $x,y\in A, a,b\in B$, the following equations hold:
\begin{equation}\mlabel{eq:defi:MP pre-Lie1}
\begin{split}
\begin{array}{ll}
    l_{A}(x)(a\circ_{B}b)=&-l_{A}(l_{B}(a)x-r_{B}(a)x)b+(l_{A}(x)a\\
    &-r_{A}(x)a)\circ_{B}b+r_{A}(r_{B}(b)x)a+a\circ_{B}(l_{A}(x)b),
\end{array}
\end{split}
\end{equation}
\begin{equation}\mlabel{eq:defi:MP pre-Lie2}
r_{A}(x)[a,b]_{B}=r_{A}(l_{B}(b)x)a-r_{A}(l_{B}(a)x)b+a\circ_{B}(r_{A}(x)b)-b\circ_{B}(r_{A}(x)a),
\end{equation}
\begin{equation}\mlabel{eq:defi:MP pre-Lie3}
    \begin{split}
\begin{array}{ll} l_{B}(a)(x\circ_{A}y)=&-l_{B}(l_{A}(x)a-r_{A}(x)a)y+(l_{B}(a)x-r_{B}(a)x)\circ_{A}y\\
    &+r_{B}(r_{A}(y)a)x+x\circ_{A}(l_{B}(a)y),
\end{array}
\end{split}
\end{equation}
\begin{equation}\mlabel{eq:defi:MP pre-Lie4}
r_{B}(a)[x,y]_{A}=r_{B}(l_{A}(y)a)x-r_{B}(l_{A}(x)a)y+x\circ_{A}(r_{B}(a)y)-y\circ_{A}(r_{B}(y)x).
\end{equation}
Then $(A,B,l_{A},r_{A},l_{B},r_{B})$ is called a \textbf{matched pair of pre-Lie algebras}.
\end{defi}

Let $(A,\circ_{A})$ and $(B,\circ_{B})$ be pre-Lie algebras, and $l_{A},r_{A}:A\rightarrow\mathrm{End}(B)$, $l_{B},r_{B}:B\rightarrow\mathrm{End}(A)$ be linear maps. Define a multiplication $\circ$ on $A\oplus B$ by
$$(x+a)\circ(y+b)=x\circ_{A}y+l_{B}(a)y+r_{B}(b)x+a\circ_{B}b+l_{A}(x)b+r_{A}(y)a,\;\;\forall x,y\in A, a,b\in B.$$
Then $(A\oplus B,\circ)$ is a pre-Lie algebra if and only if  $(A,B,l_{A},r_{A},l_{B},r_{B})$ is a matched pair of pre-Lie algebras. In this case, we denote the pre-Lie algebra structure on $A\oplus B$ by $A\bowtie^{l_{A},r_{A}}_{l_{B},r_{B}}B$ or simply $A\bowtie B$.

The relationship between  Manin triples of pre-Lie algebras with respect to the \nonsy left-invariant bilinear form $\mathcal{B}_{d}$ and matched pairs of pre-Lie algebras is characterized as follows in terms of \spec \ldends.

\begin{thm}\mlabel{thm:pre-Lie 12} \mcite{BHC}
Let $(A,\circ_{A})$ and $(A^{*},\circ_{A^{*}})$ be pre-Lie algebras. Suppose that there is a pre-Lie algebra structure $\circ$ on the direct sum  $A\oplus A^{*}$ of vector spaces such that $(A\bowtie A^{*}, A, A^*)$ is
a Manin triple of pre-Lie algebras with respect to the \nonsy left-invariant bilinear form $\mathcal{B}_{d}$.
Then there is a compatible \spec \ldend $(A\oplus A^{*},\triangleright,\triangleleft)$ of the pre-Lie algebra $(A\bowtie A^*,\circ)$ defined by
\begin{equation}\mlabel{eq:thm:pre-Lie 12 1}
\mathcal{B}_{d}((x+a^{*})\triangleleft(y+b^{*}),z+c^{*})=\mathcal{B}_{d}(x+a^{*},(z+c^{*})\circ(y+b^{*})),
\end{equation}
\begin{equation}\mlabel{eq:thm:pre-Lie 12 2}
(x+a^{*})\triangleright(y+b^{*})=(x+a^{*})\circ(y+b^{*})-(x+a^{*})\triangleleft(y+b^{*}),\;\;\forall x,y,z\in A, a^{*}, b^{*},c^*\in A^{*}.
\end{equation}
It contains the two compatible special  L-dendriform subalgebras
$(A,\triangleright_{A},\triangleleft_{A})$ and
$(A^{*},\triangleright_{A^{*}},\triangleleft_{A^{*}})$ of the
pre-Lie algebras $(A,\circ_A)$ and $(A^*,\circ_{A^*})$
respectively, such that $(A, A^{*},
\mathcal{L}^{*}_{\circ_{A}},
\mathcal{L}^{*}_{\triangleleft_{A}}, \mathcal{L}^{*}_{\circ_{A^{*}}},$ $\mathcal{L}^{*}_{\triangleleft_{A^{*}}})$
is a matched pair of pre-Lie algebras.

Conversely, suppose that
$(A, \triangleright_{A},\triangleleft_{A})$ and $(A^{*},
\triangleright_{A^{*}},\triangleleft_{A^{*}})$ are compatible
\spec \ldends of the pre-Lie algebras
$(A,\circ_{A})$ and $(A^{*},\circ_{A^{*}})$ respectively. If $(A,A^{*},\mathcal{L}^{*}_{\circ_{A}},\mathcal{L}^{*}_{\triangleleft_{A}},\mathcal{L}^{*}_{\circ_{A^{*}}},$ $
\mathcal{L}^{*}_{\triangleleft_{A^{*}}})$ is a matched pair of pre-Lie algebras, then there is a Manin triple $(A\bowtie A^*, A, A^*)$ of pre-Lie algebras with respect to the \nonsy left-invariant bilinear form $\mathcal{B}_{d}$.
Moreover, Eqs.~\meqref{eq:thm:pre-Lie 12 1} and
\meqref{eq:thm:pre-Lie 12 2} hold.
\end{thm}

The relationship between matched pairs of Rota-Baxter Lie
algebras of weight zero and matched pairs of the induced pre-Lie
algebras is given as follows.

\begin{pro}\mlabel{pro:rel:12}
  Let $(\mathfrak{g},[-,-]_{\mathfrak{g}},P)$ and $(\mathfrak{g}^{*},[-,-]_{\mathfrak{g}^{*}},Q^{*})$ be  Rota-Baxter Lie algebras of weight zero such that $Q$ is admissible to $(\mathfrak{g},[-,-]_{\mathfrak{g}},P)$ and $P^{*}$
    is admissible to $(\mathfrak{g}^{*},[-,-]_{\mathfrak{g}^{*}}$, $Q^{*})$.
Let
$(\mathfrak{g},\circ_{\mathfrak{g}})$ and
$(\mathfrak{g}^{*},\circ_{\mathfrak{g}^{*}})$ be the induced
pre-Lie algebras defined in Eq.~\meqref{eq:pro:from invariance to left invariance1}, and
$(\mathfrak{g},\triangleright_{\mathfrak{g}},\triangleleft_{\mathfrak{g}})$
and
$(\mathfrak{g}^{*},\triangleright_{\mathfrak{g}^{*}},\triangleleft_{\mathfrak{g}^{*}})$
be the compatible \spec \ldends of the pre-Lie
algebras $(\mathfrak{g},\circ_{\mathfrak{g}})$ and
$(\mathfrak{g}^{*},\circ_{\mathfrak{g}^{*}})$ defined by
Eqs.~\meqref{eq:cor:MP adm9} and \meqref{eq:cor:MP adm10}
respectively. If
$((\mathfrak{g},P),(\mathfrak{g}^{*},Q^{*}),\mathrm{ad}^{*}_{\mathfrak{g}},\mathrm{ad}^{*}_{\mathfrak{g}^{*}})$
is a matched pair of Rota-Baxter Lie algebras of weight zero, then
$(\mathfrak{g},\mathfrak{g}^{*},\mathcal{L}^{*}_{\circ_{\mathfrak{g}}}$,
$\mathcal{L}^{*}_{\triangleleft_{\mathfrak{g}}},\mathcal{L}^{*}_{\circ_{\mathfrak{g}^{*}}},\mathcal{L}^{*}_{\triangleleft_{\mathfrak{g}^{*}}})$
is a matched pair of pre-Lie algebras.
\end{pro}

\begin{proof}
By Proposition \mref{pro:sp dend cond},
$\rbrep{\mathcal{L}^{*}_{\circ_{\mathfrak{g}}}}{\mathcal{L}^{*}_{\triangleleft_{\mathfrak{g}}}}{\mathfrak{g}^{*}}$
is a representation of $(\mathfrak{g},\circ_{\mathfrak{g}})$, and
$\rbrep{\mathcal{L}^{*}_{\circ_{\mathfrak{g}^{*}}}}{\mathcal{L}^{*}_{\triangleleft_{\mathfrak{g}^{*}}}}{\mathfrak{g}}$
is a representation of
$(\mathfrak{g}^{*},\circ_{\mathfrak{g}^{*}})$. Let
$x\in\mathfrak{g}, a^{*}, b^{*}\in\mathfrak{g}^{*}$. Then we have
\begin{small}
\begin{eqnarray*}
\mathcal{L}^{*}_{\circ_{\mathfrak{g}}}(x)(a^{*}\circ_{\mathfrak{g}^{*}}b^{*})&=&\mathrm{ad}^{*}_{\mathfrak{g}}(P(x))[Q^{*}(a^{*}),b^{*}]_{\mathfrak{g}^{*}},\\
\mathcal{L}^{*}_{\circ_{\mathfrak{g}}}(\mathcal{L}^{*}_{\circ_{\mathfrak{g}^{*}}}(a^{*})x-\mathcal{L}^{*}_{\triangleleft_{\mathfrak{g}^{*}}}(a^{*})x)b^{*}&=&
\mathrm{ad}^{*}_{\mathfrak{g}}(P(\mathrm{ad}^{*}_{\mathfrak{g}^{*}}(Q^{*}(a^{*}))+\mathrm{ad}^{*}_{\mathfrak{g}^{*}}(a^{*})P(x)))b^{*}\\
&\overset{(\mref{adm cond 2})}{=}&\mathrm{ad}^{*}_{\mathfrak{g}}(\mathrm{ad}^{*}_{\mathfrak{g}^{*}}(Q^{*}(a^{*}))P(x))b^{*},\\
-(\mathcal{L}^{*}_{\circ_{\mathfrak{g}}}(x)a^{*}-\mathcal{L}^{*}_{\triangleleft_{\mathfrak{g}}}(x)a^{*})\circ_{\mathfrak{g}^{*}}b^{*}&=&-[Q^{*}
(\mathrm{ad}^{*}_{\mathfrak{g}}(P(x))a^{*}+\mathrm{ad}^{*}_{\mathfrak{g}}(x)Q^{*}(a^{*})),b^{*}]_{\mathfrak{g}^{*}}\\
&\overset{(\mref{eq:cor:adm1})}{=}&-[\mathrm{ad}^{*}_{\mathfrak{g}}(P(x))Q^{*}(a^{*}),b^{*}]_{\mathfrak{g}^{*}},\\
-\mathcal{L}^{*}_{\triangleleft_{\mathfrak{g}}}(\mathcal{L}^{*}_{\triangleleft_{\mathfrak{g}^{*}}}(b^{*})x)a^{*}&=&-\mathrm{ad}^{*}_{\mathfrak{g}}(\mathrm{ad}^{*}_{\mathfrak{g}^{*}}(b^{*})P(x))Q^{*}(a^{*}),\\
-a^{*}\circ_{\mathfrak{g}^{*}}\mathcal{L}^{*}_{\circ_{\mathfrak{g}}}(x)b^{*}&=&-[Q^{*}(a^{*}),\mathrm{ad}^{*}_{\mathfrak{g}}(P(x))b^{*}]_{\mathfrak{g}^{*}}.
\end{eqnarray*}
\end{small}

Since $(\frak g,\frak g^*,
\mathrm{ad}^{*}_{\mathfrak{g}},\mathrm{ad}^{*}_{\mathfrak{g}^{*}})$
is a matched pair of Lie algebras, we have
\begin{equation*}\mlabel{eq:mp1}
\mathrm{ad}^{*}_{\mathfrak{g}}(x)[a^{*},b^{*}]_{\mathfrak{g}^{*}}-[\mathrm{ad}^{*}_{\mathfrak{g}}(x)a^{*},b^{*}]_{\mathfrak{g}^{*}}-[a^{*},\mathrm{ad}^{*}_{\mathfrak{g}}(x)b^{*}]_{\mathfrak{g}^{*}}
+\mathrm{ad}^{*}_{\mathfrak{g}}(\mathrm{ad}^{*}_{\mathfrak{g}^{*}}(a^{*})x)b^{*}-\mathrm{ad}^{*}_{\mathfrak{g}}(\mathrm{ad}^{*}_{\mathfrak{g}^{*}}(b^{*})x)a^{*}=0.
\end{equation*}
Thus Eq.~\meqref{eq:defi:MP pre-Lie1} holds by taking $$(A,\circ_A)=(\frak g,\circ_{\frak g}), (B,\circ_B)=(\frak g^*,\circ_{\frak
g^*}),l_A=\mathcal{L}^{*}_{\circ_{\mathfrak{g}}},
r_A=\mathcal{L}^{*}_{\triangleleft_{\mathfrak{g}}},
l_B=\mathcal{L}^{*}_{\circ_{\mathfrak{g}^{*}}},
r_B=\mathcal{L}^{*}_{\triangleleft_{\mathfrak{g}^{*}}}.$$
Similarly, Eqs.~\meqref{eq:defi:MP pre-Lie2}-(\mref{eq:defi:MP
pre-Lie4}) hold.
Hence
$(\mathfrak{g},\mathfrak{g}^{*},\mathcal{L}^{*}_{\circ_{\mathfrak{g}}},\mathcal{L}^{*}_{\triangleleft_{\mathfrak{g}}},\mathcal{L}^{*}_{\circ_{\mathfrak{g}^{*}}},\mathcal{L}^{*}_{\triangleleft_{\mathfrak{g}^{*}}})$
is a matched pair of pre-Lie algebras.
\end{proof}

\subsubsection{Special \ldendbs}\
There is a bialgebra structure for \spec \ldends obtained in
\mcite{BHC}.

\begin{defi}\mlabel{defi:sp L-dend coalg}
Let $A$ be a vector space and
$\copa,\copb:A\rightarrow A\otimes A$ be linear maps. Suppose that
$\copb$ is co-antisymmetric. Set $\copc:=\copa+\copb$. If the following two
conditions hold:
\begin{equation}\mlabel{eq:defi:sp L-dend coalg1}
(\mathrm{id}\otimes \copb)\copb+(\tau\otimes
\mathrm{id})(\mathrm{id}\otimes
\copc)\copb+(\copc\otimes\mathrm{id})\copb-(\mathrm{id}\otimes
\copb)\copc=0,
\end{equation}
\begin{equation}\mlabel{eq:defi:sp L-dend coalg2}
    (\copc\otimes\mathrm{id})\copc-(\mathrm{id}\otimes\copc)\copc=(\tau\otimes\mathrm{id})(\copc\otimes\mathrm{id})\copc-(\tau\otimes\mathrm{id})(\mathrm{id}\otimes\copc)\copc,
\end{equation}
then $(A,\copa,\copb)$ is called a \textbf{special L-dendriform
coalgebra}.
\end{defi}

\begin{pro}\mlabel{pro:from coalg to alg}
Let $A$ be a finite dimensional vector space, and
$\copa,\copb:A\rightarrow A\otimes A$ be linear maps. Let
$\triangleright_{A^{*}},\triangleleft_{A^{*}}:A^{*}\otimes
A^{*}\rightarrow A^{*}$ be the linear duals of $\copa$ and $\copb$
respectively.
Then $(A,\copa,\copb)$ is a special L-dendriform coalgebra if and
only if $(A^{*}, \triangleright_{A^{*}},\triangleleft_{A^{*}})$ is
a \spec \ldend.
\end{pro}
\begin{proof}
It is obvious that $\copb$ is co-antisymmetric if and only if
$\triangleleft_{A^{*}}$ is antisymmetric. Let
$\circ_{A^{*}}:A^{*}\otimes A^{*}\rightarrow A^{*}$ be a linear
operation given by
$$a^{*}\circ_{A^{*}}b^{*}=a^{*}\triangleright_{A^{*}}b^{*}+a^{*}\triangleleft_{A^{*}}b^{*},\;\;\forall a^{*}, b^{*}\in A^{*}.$$
Then for all $x\in A, a^{*}, b^{*}, c^{*}\in A^{*}$, we have
{\small
\begin{eqnarray*}
&&\langle (\mathrm{id}\otimes \copb)\copb(x)+(\tau\otimes \mathrm{id})(\mathrm{id}\otimes \copc)\copb(x)+(\copc\otimes\mathrm{id})\copb(x)-(\mathrm{id}\otimes \copb)\copc(x), a^{*}\otimes b^{*}\otimes c^{*}\rangle\\
&&=\langle x,a^{*}\triangleleft_{A^{*}}(b^{*}\triangleleft_{A^{*}}c^{*})+b^{*}\triangleleft_{A^{*}}(a^{*}\circ_{A^{*}}c^{*})+(a^{*}\circ_{A^{*}}b^{*})\triangleleft_{A^{*}}c^{*}-a^{*}\circ_{A^{*}}(b^{*}\triangleleft_{A^{*}}c^{*})\rangle,\\
&&\langle (\copc\otimes\mathrm{id})\copc(x)-(\mathrm{id}\otimes\copc)\copc(x)-(\tau\otimes\mathrm{id})(\copc\otimes\mathrm{id})\copc(x)+(\tau\otimes\mathrm{id})(\mathrm{id}\otimes\copc)\copc(x), a^{*}\otimes b^{*}\otimes c^{*}\rangle\\
&&=\langle x, (a^{*}\circ_{A^{*}}b^{*})\circ_{A^{*}}c^{*}-a^{*}\circ_{A^{*}}(b^{*}\circ_{A^{*}}c^{*})-(b^{*}\circ_{A^{*}}a^{*})\circ_{A^{*}}c^{*}+b^{*}\circ_{A^{*}}(a^{*}\circ_{A^{*}}c^{*})\rangle.
\end{eqnarray*}}
Hence the conclusion follows by Proposition \mref{pro:equivalence}.
\end{proof}

\begin{defi}\mlabel{defi:sp L-dend bialg}
A \textbf{\spec \ldendb}
$(A,\triangleright,\triangleleft,\copa,\copb)$ consists of a \spec
\ldend $(A,\triangleright,\triangleleft)$, a special L-dendriform
coalgebra $(A,\copa,\copb)$ such that the following conditions
hold:
\begin{equation}\mlabel{eq:defi:sp L-dend bialg1}
\copc(x\circ
y)-(\mathrm{id}\otimes\mathcal{R}_{\circ}(y))\copa(x)+(\mathcal{L}_{\triangleleft}(y)\otimes
\mathrm{id})\copb(x)-(\mathcal{L}_{\triangleright}(x)\otimes
\mathrm{id}+\mathrm{id}\otimes\mathcal{L}_{\circ}(x))\copc(y)=0,
\end{equation}
\begin{equation}\mlabel{eq:defi:sp L-dend bialg2}
(\tau-\mathrm{id})((\mathrm{id}\otimes\mathcal{L}_{\triangleleft}(x))\copc(y)-(\mathrm{id}\otimes\mathcal{L}_{\triangleleft}(y))\copc(x)-\copc(x\triangleleft
y))=0,
\end{equation}
\begin{equation}\mlabel{eq:defi:sp L-dend bialg3}
\copb([x,y])
+(\mathcal{L}_{\circ}(y)\otimes\mathrm{id}+\mathrm{id}\otimes\mathcal{L}_{\circ}(y))\copb(x)-(\mathcal{L}_{\circ}(x)\otimes\mathrm{id}+\mathrm{id}\otimes\mathcal{L}_{\circ}(x))\copb(y)=0,
\end{equation}
for all $x,y\in A$, where $\copc=\copa+\copb$.
\end{defi}

\begin{rmk}
This notion of a \spec \ldendb is equivalent to the one introduced in \mcite{BHC} which is given in terms of the operations in the dual space. Indeed, in~\mcite{BHC} the compatible conditions are given by four equations in Theorem~3.16 there. The first and second equations coincide with Eqs.~\meqref{eq:defi:sp L-dend bialg3} and Eq.~\meqref{eq:defi:sp L-dend bialg1} respectively, while the third and fourth equations are equivalent to Eqs.~\meqref{eq:defi:sp L-dend bialg2} and \meqref{eq:defi:sp L-dend bialg1} respectively.
\end{rmk}

\begin{thm}\mlabel{rkm:sp L-dend23} {\rm \mcite{BHC}}
Let $(A,\triangleright_{A},\triangleleft_{A})$ be a \spec \ldend. Suppose that
    there is a \spec \ldend structure $(A^{*},\triangleright_{A^{*}},\triangleleft_{A^{*}})$ on the dual space $A^{*}$.
    Let $(A,\circ_{A})$ and $(A^{*}, \circ_{A^{*}})$ be the sub-adjacent pre-Lie algebras of $(A,\triangleright_{A},\triangleleft_{A})$ and $(A^{*},\triangleright_{A^{*}},\triangleleft_{A^{*}})$ respectively,
    and $\copa, \copb:A\rightarrow A\otimes A$ be the linear duals of $\triangleright_{A^{*}}, \triangleleft_{A^{*}}$ respectively. Then the sextuple $(A,A^{*},\mathcal{L}^{*}_{\circ_{A}},\mathcal{L}^{*}_{\triangleleft_{A}},\mathcal{L}^{*}_{\circ_{A^{*}}},\mathcal{L}^{*}_{\triangleleft_{A^{*}}})$
is a matched pair of pre-Lie algebras if and only if
$(A,\triangleright_{A},\triangleleft_{A},\copa,\copb)$ is a \spec
\ldendb.
\end{thm}

Combining Theorems ~\mref{thm:pre-Lie 12} and~\mref{rkm:sp L-dend23}, we obtain the following conclusion.

\begin{cor}\mlabel{cor:equ}
With the conditions in Theorem~\mref{rkm:sp L-dend23}, the
following statements are equivalent.
\begin{enumerate}
        \item There is a Manin triple $(A\bowtie A^*, A, A^*)$ of pre-Lie
algebras with respect to the \nonsy
left-invariant bilinear form $\mathcal{B}_{d}$ such that the
induced \spec \ldend $(A\oplus
A^{*},\triangleright,\triangleleft)$  defined by
Eqs.~\meqref{eq:thm:pre-Lie 12 1} and \meqref{eq:thm:pre-Lie 12 2}
includes $(A,\triangleright_{A},\triangleleft_{A})$ and
$(A^{*},\triangleright_{A^{*}},\triangleleft_{A^{*}})$ as
subalgebras.
        \item $(A,A^{*},\mathcal{L}^{*}_{\circ_{A}},\mathcal{L}^{*}_{\triangleleft_{A}},\mathcal{L}^{*}_{\circ_{A^{*}}},\mathcal{L}^{*}_{\triangleleft_{A^{*}}})$ is a matched pair of pre-Lie algebras.
        \item $(A,\triangleright_{A},\triangleleft_{A},\copa,\copb)$ is a \spec \ldendb.
    \end{enumerate}
\end{cor}

\begin{pro}\mlabel{pro:condition}
    Let $(\mathfrak{g},[-,-]_{\mathfrak{g}},P)$ and $(\mathfrak{g}^{*},[-,-]_{\mathfrak{g}^{*}},Q^{*})$ be  Rota-Baxter Lie algebras of weight zero such that $Q$ is admissible to $(\mathfrak{g},[-,-]_{\mathfrak{g}},P)$ and $P^{*}$
    is admissible to $(\mathfrak{g}^{*},[-,-]_{\mathfrak{g}^{*}}$, $Q^{*})$.
   Let
$(\mathfrak{g},\circ_{\mathfrak{g}})$ and
$(\mathfrak{g}^{*},\circ_{\mathfrak{g}^{*}})$ be the induced
pre-Lie algebras defined by Eq.~\meqref{eq:pro:from invariance to left invariance1}, and let
$(\mathfrak{g},\triangleright_{\mathfrak{g}},\triangleleft_{\mathfrak{g}})$
and
$(\mathfrak{g}^{*},\triangleright_{\mathfrak{g}^{*}},\triangleleft_{\mathfrak{g}^{*}})$
be the compatible \spec \ldends of the pre-Lie
algebras $(\mathfrak{g},\circ_{\mathfrak{g}})$ and
$(\mathfrak{g}^{*},\circ_{\mathfrak{g}^{*}})$ defined by
Eqs.~\meqref{eq:cor:MP adm9} and \meqref{eq:cor:MP adm10}
respectively. Let
$\delta,
\copa,\copb:\mathfrak{g}\rightarrow \mathfrak{g}\otimes
\mathfrak{g}$ be the linear duals of
$[-,-]_{\mathfrak{g}^{*}}, \triangleright_{\mathfrak{g}^{*}}$ and
$\triangleleft_{\mathfrak{g}^{*}}$ respectively. Then
$(\mathfrak{g},\triangleright_{\mathfrak{g}},\triangleleft_{\mathfrak{g}},\copa,\copb)$
is a \spec \ldendb if and only if the following equations hold:
\begin{small}
\begin{equation}\mlabel{eq:pro:condition1}
\begin{split}
\begin{array}{ll}
&(Q\otimes\mathrm{id})(\delta([P(x),y]_{\mathfrak{g}}+(\mathrm{id}\otimes\mathrm{ad}_{\mathfrak{g}}(y)+\mathrm{ad}_{\mathfrak{g}}(y)\otimes\mathrm{id})\delta(P(x))\\ &-(\mathrm{id}\otimes\mathrm{ad}_{\mathfrak{g}}(P(x))+\mathrm{ad}_{\mathfrak{g}}(P(x))\otimes\mathrm{id})\delta(y))=0,
\end{array}
\end{split}
\end{equation}
\begin{equation}\mlabel{eq:pro:condition2}
(Q\otimes Q)(\delta([x,y]_{\mathfrak{g}}+(\mathrm{id}\otimes\mathrm{ad}_{\mathfrak{g}}(y)+\mathrm{ad}_{\mathfrak{g}}(y)\otimes\mathrm{id})\delta(x)-(\mathrm{id}\otimes\mathrm{ad}_{\mathfrak{g}}(x)+\mathrm{ad}_{\mathfrak{g}}(x)\otimes\mathrm{id})\delta(y))=0,
\end{equation}
\begin{equation}\mlabel{eq:pro:condition3}
\begin{split}
    \begin{array}{ll}
\delta([P(x),P(y)]_{\mathfrak{g}})=& (\mathrm{id}\otimes\mathrm{ad}_{\mathfrak{g}}(P(x)) +\mathrm{ad}_{\mathfrak{g}}(P(x))\otimes\mathrm{id})\delta(P(y))\\ & -(\mathrm{id}\otimes\mathrm{ad}_{\mathfrak{g}}(P(y))+\mathrm{ad}_{\mathfrak{g}}(P(y))\otimes\mathrm{id})\delta(P(x)),
\end{array}
\end{split}
\end{equation}
\end{small}
for all $x,y\in\mathfrak{g}$. In particular,  if
$(\mathfrak{g},[-,-]_{\mathfrak{g}},P,\delta,Q)$ is  a Rota-Baxter
Lie bialgebra of weight zero, then
$(\mathfrak{g},\triangleright_{\mathfrak{g}}$,$\triangleleft_{\mathfrak{g}}$,$\copa,\copb)$
is a \spec \ldendb.
\end{pro}

\begin{proof}
It is obvious that $(\mathfrak{g},\copa,\copb)$ is a special
L-dendriform coalgebra.
Let
$\copc$ be the linear dual of $\circ_{\mathfrak{g}^{*}}$. Let
$x,y\in \frak g$. Then by Eqs.~\meqref{eq:pro:from invariance to
left invariance1}, \meqref{eq:cor:MP adm9} and \meqref{eq:cor:MP
adm10},  we have
\begin{eqnarray}\mlabel{eq:pro:condition4}
&\copc(x)=(Q\otimes\mathrm{id})\delta(x),\;\; \copb(x)=-\delta
(P(x)),\;\; \copa(x)=(Q\otimes\mathrm{id})\delta(x)+\delta (P(x)),
&
\\
\mlabel{eq:pro:condition5}
&   \begin{split}
        \begin{array}{l}
\mathcal{L}_{\circ_{\mathfrak{g}}}(x)=\mathrm{ad}_{\mathfrak{g}}(P(x)), \mathcal{R}_{\circ_{\mathfrak{g}}}(x)=-\mathrm{ad}_{\mathfrak{g}}(x)P, \\ \mathcal{L}_{\triangleleft_{\mathfrak{g}}}(x)
=-Q\mathrm{ad}_{\mathfrak{g}}(x), \mathcal{L}_{\triangleright_{\mathfrak{g}}}
(x)=\mathrm{ad}_{\mathfrak{g}}(P(x))+Q\mathrm{ad}_{\mathfrak{g}}(x).
\end{array}
\end{split}
&
\end{eqnarray}
Therefore we have
\begin{eqnarray*}
\copc(x\circ_{\mathfrak{g}}y)&=&(Q\otimes\mathrm{id})\delta([P(x),y]_{\mathfrak{g}}),\\
-(\mathrm{id}\otimes\mathcal{R}_{\circ_{\mathfrak{g}}})\copa(x)
&=&(\mathrm{id}\otimes\mathrm{ad}_{\mathfrak{g}}(y)P)((Q\otimes\mathrm{id})\delta(x)+\delta(P(x)))\\
&=&(\mathrm{id}\otimes\mathrm{ad}_{\mathfrak{g}}(y))((Q\otimes P)\delta(x)+(\mathrm{id}\otimes P)\delta(P(x)))\\
&\overset{(\mref{eq:defi:RB Lie bialg1})}{=}&(\mathrm{id}\otimes\mathrm{ad}_{\mathfrak{g}}(y))(Q\otimes\mathrm{id})\delta(P(x))=(Q\otimes\mathrm{id})(\mathrm{id}\otimes\mathrm{ad}_{\mathfrak{g}}(y))\delta(P(x)),\\
(\mathcal{L}_{\triangleleft_{\mathfrak{g}}}(y)\otimes\mathrm{id})\copb(x)&=&(Q\otimes\mathrm{id})(\mathrm{ad}_{\mathfrak{g}}(y)\otimes\mathrm{id})\delta(P(x)),\\
-(\mathcal{L}_{\triangleright_{\mathfrak{g}}}(x)\otimes\mathrm{id})\copc(y)
&=&-(\mathrm{ad}_{\mathfrak{g}}(P(x))\otimes\mathrm{id})(Q\otimes\mathrm{id})\delta(y)-(Q\mathrm{ad}_{\mathfrak{g}}(x)\otimes\mathrm{id})(Q\otimes\mathrm{id})\delta(y)\\
&\overset{(\mref{eq:cor:adm1})}{=}&-(Q\otimes\mathrm{id})(\mathrm{ad}_{\mathfrak{g}}(P(x))\otimes\mathrm{id})\delta(y),\\
-(\mathrm{id}\otimes\mathcal{L}_{\circ_{\mathfrak{g}}})\copc(y)&=&-(Q\otimes\mathrm{id})(\mathrm{id}\otimes\mathrm{ad}_{\mathfrak{g}}(P(x)))\delta(y).
\end{eqnarray*}
Thus Eq.~\meqref{eq:defi:sp L-dend bialg1} is equivalent to
Eq.~\meqref{eq:pro:condition1}. Similarly,
Eq.~\meqref{eq:defi:sp L-dend bialg2} and Eq.~\meqref{eq:defi:sp
    L-dend bialg3} are equivalent to
Eq.~\meqref{eq:pro:condition2} and Eq.~\meqref{eq:pro:condition3} respectively. In particular,  if
$(\mathfrak{g},[-,-]_{\mathfrak{g}},P,\delta,Q)$ is  a Rota-Baxter
Lie bialgebra of weight zero, then Eq.~\meqref{eq:defi:Lie
bialg1} holds. Therefore Eqs.~(\mref{eq:pro:condition1})-(\mref{eq:pro:condition3}) hold.
Thus
$(\mathfrak{g},\triangleright_{\mathfrak{g}}$,$\triangleleft_{\mathfrak{g}}$,$\copa,\copb)$
is a \spec \ldendb.
\end{proof}

The relations among the various structures in
this section can be summarized in the following commutative diagram.
\vspace{-.3cm}
    \begin{equation*}
        \xymatrix{
         \txt{Manin triples of  Rota-Baxter Lie algebras}
            \ar@{=>}[r]^-{{\rm Cor.}~\mref{cor:from invariance to left invariance}}_-{\lambda=0}
            \ar@{<=>}[d]^-{{\rm Thm.}~\mref{thm:123 RB Lie}}
             & \txt{Manin triples of pre-Lie algebras} \ar@{<=>}[d]^-{{\rm Cor.}~\mref{cor:equ}}\\
             \txt{matched pairs of  Rota-Baxter Lie algebras}
             \ar@{<=>}[d]^-{{\rm Thm.}~\mref{thm:123 RB Lie}}
              \ar@{=>}[r]^-{{\rm Prop.}~\mref{pro:rel:12}}_-{\lambda=0}
            & \txt{matched pairs of   pre-Lie algebras}
             \ar@{<=>}[d]^-{{\rm Cor.}~\mref{cor:equ}}
             \\
             \txt{Rota-Baxter   Lie bialgbras}
            \ar@{=>}[r]^-{{\rm Prop.}~\mref{pro:condition} }_-{\lambda=0}
            & \txt{special L-dendriform   bialgebras}}
    \end{equation*}

\section{Coboundary Rota-Baxter Lie bialgebras, admissible CYBEs and the induced \spec \ldendbs }\mlabel{sec:4}

In the last section of the paper, we study the coboundary Rota-Baxter Lie bialgebras, leading to the
notion of admissible CYBE in Rota-Baxter Lie algebras,
whose antisymmetric solutions can be used to construct Rota-Baxter Lie bialgebras.
Furthermore, the notions of $\mathcal{O}$-operators on Rota-Baxter Lie algebras and Rota-Baxter pre-Lie algebras
are introduced to produce antisymmetric solutions of the admissible CYBE, and hence Rota-Baxter Lie bialgebras. When the weight of the Rota-Baxter operator is zero, we study the induced \spec \ldendbs from these Rota-Baxter Lie bialgebras and thus give the construction of special L-dendriform
 bialgebras from antisymmetric solutions of the admissible CYBE in Rota-Baxter Lie
algebras of weight zero.  In particular, both Rota-Baxter Lie
algebras of weight zero and Rota-Baxter pre-Lie algebras of weight zero can be used to construct \spec \ldendbs.

\subsection{Coboundary Rota-Baxter Lie bialgebras and the induced \spec \ldendbs}\mlabel{sec:4.1}\

Recall \mcite{CP} that a Lie bialgebra $(\mathfrak{g},[-,-],\delta)$ is called \textbf{coboundary }if there exists an $r\in \mathfrak{g}\otimes \mathfrak{g}$ such that
\begin{equation}\mlabel{eq:LieCob}
    \delta(x)\coloneqq \delta_{r}(x)\coloneqq (\mathrm{ad}(x)\otimes \mathrm{id}+\mathrm{id}\otimes\mathrm{ad}(x))r, \;\;\forall x\in \mathfrak{g}.
\end{equation}

For Rota-Baxter Lie bialgebras we give a similar notion.

\begin{defi}
    A Rota-Baxter Lie bialgebra  $(\mathfrak{g},[-,-],P,\delta,Q)$ is called \textbf{coboundary} if there exists an $r\in \mathfrak{g}\otimes \mathfrak{g}$ such that Eq.~\meqref{eq:LieCob} holds.
\end{defi}

Let $(\mathfrak{g},[-,-])$ be a Lie algebra, and $r=\sum\limits_{i}a_{i}\otimes b_{i}\in\mathfrak{g}\otimes\mathfrak{g}$.
Let  $\delta:\mathfrak{g}\rightarrow
\mathfrak{g}\otimes \mathfrak{g}$ be a linear map defined by Eq.~\meqref{eq:LieCob}. Then
$\delta$ satisfies Eq.~\meqref{eq:defi:Lie bialg1} automatically. Moreover,
by \mcite{CP}, $\delta$ makes $(\mathfrak{g},\delta)$ into a Lie coalgebra
such that $(\mathfrak{g},[-,-],\delta)$ is a Lie bialgebra if and only if for
all $x\in \mathfrak{g},$
\begin{equation}\mlabel{eq:CYBE1}
    (\mathrm{ad}(x)\otimes \mathrm{id}+\mathrm{id}\otimes\mathrm{ad}(x))(r+\tau(r))=0,
\end{equation}
\begin{equation}\mlabel{eq:CYBE2}
    (\mathrm{ad}(x)\otimes \mathrm{id}\otimes \mathrm{id}+\mathrm{id}\otimes\mathrm{ad}(x)\otimes \mathrm{id}+\mathrm{id}\otimes \mathrm{id}\otimes\mathrm{ad}(x))([r_{12},r_{13}]+[r_{12},r_{23}]+[r_{13},r_{23}])=0,
\end{equation}
where
$$[r_{12},r_{13}]\coloneqq \sum_{i,j}[a_{i},a_{j}]\otimes b_{i}\otimes b_{j}, [r_{12} ,r_{23}]\coloneqq \sum_{i,j}a_{i}\otimes [b_{i}, a_{j}]\otimes b_{j}, [r_{13},r_{23}]\coloneqq \sum_{i,j}a_{i}\otimes a_{j}\otimes[b_{i},b_{j}].$$

Hence in order for $(\mathfrak{g},[-,-],P,\delta,Q)$ to be a
Rota-Baxter Lie bialgebra, we only need to further require that $(\mathfrak{g}^{*},\delta^{*},Q^{*})$ is a $P^{*}$-admissible Rota-Baxter Lie algebra, that is, $(\mathfrak{g},\delta,Q)$ is a Rota-Baxter Lie coalgebra and Eq.~\meqref{eq:defi:RB Lie bialg1} holds.

\begin{pro}\mlabel{thm:coboundary condition}
Let $(\mathfrak{g},[-,-],P)$ be a $Q$-admissible Rota-Baxter Lie algebra  of weight $\lambda$ and $r\in \mathfrak{g}\otimes \mathfrak{g}$. Define a linear map $\delta:\mathfrak{g}\rightarrow\mathfrak{g}\otimes\mathfrak{g}$
by Eq.~\meqref{eq:LieCob}. Suppose that $\delta^{*}$ defines a Lie algebra structure on $\mathfrak{g}^{*}$. Then the following conclusions hold.
\begin{enumerate}
    \item Eq.~\meqref{eq:defi:RB Lie co1} holds if and only if for all $x\in\mathfrak{g}$,
    \begin{small}
        \begin{equation}\mlabel{eq:thm:coboundary condition 1}
      \begin{split}
\begin{array}{l}        (\mathrm{id}\otimes Q(\mathrm{ad}(x))-\mathrm{id}\otimes\mathrm{ad}(Q(x)))(Q\otimes\mathrm{id}-\mathrm{id}\otimes P)(r)\\
        +(Q(\mathrm{ad}(x))\otimes\mathrm{id}-\mathrm{ad}(Q(x))\otimes\mathrm{id})(P\otimes\mathrm{id}-\mathrm{id}\otimes Q)(r)=0.
\end{array}
\end{split}
\end{equation}
    \end{small}
\item Eq.~\meqref{eq:defi:RB Lie bialg1} holds if and only if for all $x\in\mathfrak{g}$,
\begin{small}
\begin{equation}\mlabel{eq:thm:coboundary condition 2}
\begin{split}
    \begin{array}{l}
(\mathrm{id}\otimes\mathrm{ad}(P(x))+\mathrm{ad}(P(x))\otimes\mathrm{id}+\mathrm{id}\otimes Q(\mathrm{ad}(x))\\
-P(\mathrm{ad}(x))\otimes\mathrm{id}  +\lambda\mathrm{id}\otimes\mathrm{ad}(x))(P\otimes\mathrm{id}-\mathrm{id}\otimes Q)(r)=0.
\end{array}
\end{split}
\end{equation}
\end{small}
\end{enumerate}
\end{pro}

\begin{proof}
The conclusion is obtained by following the proof of \cite[Theorem 4.3]{BGM}.
\end{proof}

\begin{thm}
Let $(\mathfrak{g},[-,-],P)$ be a $Q$-admissible Rota-Baxter Lie algebra and $r\in \mathfrak{g}\otimes\mathfrak{g}$.  Define a linear map $\delta:\frak g\rightarrow \frak g\otimes \frak g $ by Eq.~\meqref{eq:LieCob}. Then $(\mathfrak{g},[-,-],P,\delta,Q)$ is a Rota-Baxter Lie bialgebra if and only if Eqs.~\meqref{eq:CYBE1}-\meqref{eq:thm:coboundary condition 2} hold.
\end{thm}

\begin{proof} By the assumption, $(\mathfrak{g},[-,-],P,\delta,Q)$ is a
Rota-Baxter Lie bialgebra if and only if $(\mathfrak{g}^{*}$, $\delta^{*}$, $Q^{*})$ is a $P^{*}$-admissible Rota-Baxter Lie algebra. By Proposition~\mref{thm:coboundary condition} and the results before it,
the latter holds if and only if Eqs.~\meqref{eq:CYBE1}-\meqref{eq:thm:coboundary condition 2} hold.
\end{proof}

In particular, we have the following conclusion.

\begin{cor}\mlabel{cor:equation}
    Let $(\mathfrak{g},[-,-],P)$ be a $Q$-admissible Rota-Baxter Lie algebra and $r\in \mathfrak{g}\otimes\mathfrak{g}$.  Define a linear map $\delta:\frak g\rightarrow \frak g\otimes \frak g $ by Eq.~\meqref{eq:LieCob}. Then $(\mathfrak{g},[-,-],P,\delta,Q)$ is a Rota-Baxter Lie bialgebra if  Eq.~\meqref{eq:CYBE1} and the following equations hold:
    \begin{equation}\mlabel{eq:cor:equation1}
        [r_{12},r_{13}]+[r_{12},r_{23}]+[r_{13},r_{23}]=0,
    \end{equation}
    \begin{equation}\mlabel{eq:cor:equation2}
        (P\otimes\mathrm{id}-\mathrm{id}\otimes Q)(r)=0,
    \end{equation}
    \begin{equation}\mlabel{eq:cor:equation3}
    (Q\otimes\mathrm{id}-\mathrm{id}\otimes P)(r)=0.
    \end{equation}
\end{cor}

Eq.~\meqref{eq:cor:equation1} is just the well-known {\bf classical Yang-Baxter equation (CYBE)} in $\frakg$~\mcite{CP}. In view of this, Corollary~\mref{cor:equation} suggests the following variation of the CYBE.

\begin{defi}
    Let $(\mathfrak{g},[-,-],P)$ be a Rota-Baxter Lie algebra. Suppose that $r\in \mathfrak{g}\otimes \mathfrak{g}$ and $Q:\mathfrak{g}\rightarrow \mathfrak{g}$ is a linear map. Then Eq.~\meqref{eq:cor:equation1} with conditions Eq.~\meqref{eq:cor:equation2} and Eq.~\meqref{eq:cor:equation3} is called the \textbf{$Q$-admissible classical Yang-Baxter equation}  in  $(\mathfrak{g},[-,-],P)$ or simply the \textbf{$Q$-admissible CYBE}.
\end{defi}

Note that if $r$ is antisymmetric (that is, $r=-\tau(r)$), then Eq.~\meqref{eq:cor:equation2} holds if and only if  Eq.~\meqref{eq:cor:equation3} holds.

\begin{pro}\mlabel{pro:coboundary RB}
Let $(\mathfrak{g},[-,-],P)$ be a $Q$-admissible Rota-Baxter Lie algebra and $r\in \mathfrak{g}\otimes \mathfrak{g}$ be an antisymmetric solution of the $Q$-admissible CYBE in $(\mathfrak{g},[-,-],P)$. Then $(\mathfrak{g},[-,-],P,\delta ,Q)$ is a coboundary Rota-Baxter Lie bialgebra, where the linear map $\delta=\delta_{r}$ is defined by Eq.~\meqref{eq:LieCob}.
\end{pro}

\begin{proof}
It follows from Corollary~\mref{cor:equation} immediately.
\end{proof}

On the other hand, there is a similar ``coboundary" construction of \spec \ldendbs considered in \mcite{BHC}.

\begin{pro} {\rm \mcite{BHC}}\mlabel{pro:cobound}
Let $(A,\triangleright,\triangleleft)$ be a \spec \ldend and
$(A,\circ)$ be the sub-adjacent pre-Lie algebra. Let
$r=\sum\limits_{i}a_{i}\otimes b_{i}\in A\otimes A$ be
antisymmetric. Define linear maps $\copa,\copb:A\rightarrow
A\otimes A$ by
\begin{equation}\mlabel{eq:cob sp}
\copa(x)=(\mathcal{L}_{\triangleright}(x)\otimes\mathrm{id}+\mathrm{id}\otimes\mathrm{ad}(x))(r),
\copb(x)=(\mathcal{L}_{\circ}(x)\otimes\mathrm{id}+\mathrm{id}\otimes\mathcal{L}_{\circ}(x))(-r),
\forall x\in A.
\end{equation}
Then $(A,\triangleright,\triangleleft,\copa,\copb)$ is a \spec
\ldendb if $r$ satisfies
\begin{equation}\mlabel{eq:sp L-dend equation}
r_{12}\triangleleft r_{13}=r_{12}\circ r_{23}+r_{13}\circ r_{23},
\end{equation}
\vspace{-.2cm}
where
\vspace{-.2cm}
$$r_{12}\triangleleft r_{13}=\sum_{i,j}a_{i}\triangleleft a_{j}\otimes b_{i}\otimes b_{j}, r_{12}\circ r_{23}=\sum_{i,j} a_{i}\otimes b_{i}\circ a_{j}\otimes b_{j}, r_{13}\circ r_{23}=\sum_{i,j} a_{i}\otimes a_{j}\otimes b_{i}\circ b_{j}.$$
\end{pro}
\vspace{-.2cm}
\begin{rmk}
In fact, \mcite{BHC} uses $-r$ instead of $r$ to define $\copa$
and $\copb$ in Eq.~\meqref{eq:cob sp}. We change the sign of $r$
in order to be consistent with another construction given in the
following Proposition~\mref{pro:same}.
\end{rmk}

\begin{pro}\mlabel{pro:cob sp}
Let $(\mathfrak{g},[-,-],P)$ be a Rota-Baxter Lie algebra of weight zero and $Q:\frak g\rightarrow \frak g$ be a linear map which is admissible to $(\mathfrak{g},[-,-],P)$.
Let $(\mathfrak{g},\circ)$ be the induced pre-Lie algebra and $(\mathfrak{g},\triangleright,\triangleleft)$ be the compatible special
\ldend of $(\frak g,\circ)$, where $\circ$, $\triangleleft$ and $\triangleright$ are
 defined by Eqs.~\meqref{eq:pro:from invariance to left invariance1}, \meqref{eq:left} and \meqref{eq:right} respectively, that is,
$$x\circ y=[P(x),y], \ \ x\triangleleft y=-Q([x,y]),\ \  x\triangleright y=x\circ y-x\triangleleft y,\ \ \forall x,y\in\mathfrak{g}.$$
If $r\in
\mathfrak{g}\otimes\mathfrak{g}$ is a solution of the $Q$-admissible CYBE in $(\mathfrak{g},[-,-],P)$, then $r$
satisfies Eq.~\meqref{eq:sp L-dend equation}.
\end{pro}

\begin{proof} Let $r=\sum\limits_{i}a_{i}\otimes b_{i}\in \frak g\otimes \frak g$. By Eq.~\meqref{eq:cor:equation2}, we have
\vspace{-.2cm}
\begin{eqnarray*}
&&r_{12}\circ r_{23}+r_{13}\circ r_{23}-r_{12}\triangleleft r_{13}\\
&&=\sum_{i,j}a_{i}\otimes [P(b_{i}),a_{j}]\otimes b_{j}+a_{i}\otimes a_{j}\otimes[P(b_{i}),b_{j}]+Q([a_{i},a_{j}])\otimes b_{i}\otimes b_{j}\\
&&=\sum_{i,j}Q(a_{i})\otimes[b_{i},a_{j}]\otimes b_{j}+Q(a_{i})\otimes a_{j}\otimes [b_{i},b_{j}]+Q([a_{i},a_{j}])\otimes b_{i}\otimes b_{j}\\
&&=(Q\otimes\mathrm{id}\otimes\mathrm{id})([r_{12},r_{13}]+[r_{12},r_{23}]+[r_{13},r_{23}]).
\end{eqnarray*}
Hence the conclusion follows.
\end{proof}

Let $(\mathfrak{g},[-,-],P)$ be a Rota-Baxter Lie algebra of
weight zero and $Q:\frak g\rightarrow \frak g$ be a linear map
which is admissible to $(\mathfrak{g},[-,-],P)$. An antisymmetric
solution $r\in\mathfrak{g}\otimes\mathfrak{g}$ of the
$Q$-admissible CYBE in $(\mathfrak{g},[-,-],P)$ constructs a \spec \ldendb in two ways. On the one hand, by Proposition
\mref{pro:coboundary RB}, there is a coboundary     Rota-Baxter
Lie bialgebra $(\mathfrak{g},[-,-],P,\delta,Q)$ of weight zero,
where the linear map $\delta=\delta_{r}$ is defined by
Eq.~\meqref{eq:LieCob}. Thus by Proposition \mref{pro:condition},
there is a \spec \ldendb
$(\mathfrak{g},\triangleright,\triangleleft,\copa,\copb)$, where
$\triangleright,\triangleleft,\copa,\copb$ are given by
Eqs.~\meqref{eq:cor:MP adm9} and~\meqref{eq:pro:condition4}
respectively. On the other hand, by Proposition~\mref{pro:cob sp},
$r$ satisfies Eq.~\meqref{eq:sp L-dend equation}. Hence by
Proposition~\mref{pro:cobound}, there is a \spec \ldendb
$(\mathfrak{g},\triangleright,\triangleleft,\copa',\copb')$, where
$\triangleright,\triangleleft,\copa',\copb'$ are given by
Eqs.~\meqref{eq:cor:MP adm9} and~\meqref{eq:cob sp} respectively.
The two constructions turn out to be the same.

\begin{pro}\mlabel{pro:same}
With the above notations, the \spec \ldendbs
$(\mathfrak{g},\triangleright$, $\triangleleft$, $\copa$, $\copb)$
and $(\mathfrak{g},\triangleright,\triangleleft,\copa',\copb')$
 coincide.
\end{pro}

\begin{proof}
Let the sub-adjacent Lie algebra of $(\mathfrak{g},\circ)$ be $(\mathfrak{g},[-,-]')$.
Let $r=\sum\limits_{i}a_{i}\otimes b_{i}\in\mathfrak{g}\otimes\mathfrak{g}$ and $x\in\mathfrak{g}$. Then
\begin{small}
\begin{eqnarray*}
\copb(x)&=&-\delta(P(x))=-(\mathrm{ad}(P(x))\otimes\mathrm{id}+\mathrm{id}\otimes\mathrm{ad}(P(x)))r=-\sum_{i}([P(x),a_{i}]\otimes b_{i}+a_{i}\otimes [P(x), b_{i}])\\
&=&-\sum_{i}(x\circ a_{i}\otimes b_{i}+a_{i}\otimes x\circ
b_{i})=-(\mathcal{L}_{\circ}(x)\otimes\mathrm{id}+\mathrm{id}\otimes\mathcal{L}_{\circ}(x))r=\copb'(x),
\end{eqnarray*}
\end{small}
\vspace{-.2cm}
and similarly, for $\mathrm{ad}'(x)y=[x,y]',\;\forall x,y\in\mathfrak{g}$, we have
$$\copa(x)=(Q\otimes\mathrm{id})\delta(x)+\delta (P(x))=
(\mathcal{L}_{\triangleright}(x)\otimes\mathrm{id}+\mathrm{id}\otimes\mathrm{ad}'(x))r=\copa'(x),$$
showing that the two \spec \ldendbs coincide.
\end{proof}
\vspace{-.2cm}
\subsection{Admissible CYBEs, $\mathcal{O}$-operators on Rota-Baxter Lie algebras and Rota-Baxter pre-Lie algebras}\mlabel{sec:4.2}

We first give operator forms of the antisymmetric solutions of the $Q$-admissible CYBE.
For a vector space $V$, the isomorphism $V\otimes V\cong {\rm Hom}(V^*,V)$ identifies an $r\in V\otimes V$ with a linear map $T_r:V^*\to V$. Thus for $r=\sum\limits_{i}u_{i}\otimes v_{i}$, the corresponding map $T_r$ is
\begin{equation}\mlabel{eq:identify}
   T_r:V^{*}\rightarrow V,\;\; T_r(u^{*})=\sum_{i}\langle u^{*}, u_{i}\rangle v_{i},\;\;\forall u^{*}\in V^{*}.
\end{equation}
\vspace{-.6cm}
\begin{thm}\mlabel{thm:solution}
    Let $(\mathfrak{g},[-,-],P)$ be a Rota-Baxter Lie algebra and $r\in \mathfrak{g}\otimes \mathfrak{g}$ be antisymmetric. Let $Q:\mathfrak{g}\rightarrow\mathfrak{g}$ be a linear map. Then $r$ is a solution of the $Q$-admissible CYBE in $(\mathfrak{g},[-,-],P)$ if and only if $T_r$ satisfies
    \begin{equation}\mlabel{eq:thm:solution1}
    [T_r(a^{*}),T_r(b^{*})]=T_r(\mathrm{ad}^{*}(T_r(a^{*}))b^{*}-\mathrm{ad}^{*}(T_r(b^{*}))a^{*}),\;\;\forall a^{*}, b^{*}\in \mathfrak{g}^{*},
    \end{equation}
\begin{equation}\mlabel{eq:thm:solution2}
    PT_r=T_rQ^{*}.
\end{equation}
\end{thm}
\vspace{-.2cm}
\begin{proof}
The proof follows the same argument as in the proof of \cite[Theorem 4.12]{BGM}.
\end{proof}

Then it is natural to introduce the following notion.

\begin{defi}\mlabel{defi:O-operator}
Let $(\mathfrak{g},[-,-],P)$ be a Rota-Baxter Lie algebra, $\mrep{\rho}{V}$ be a representation of $(\mathfrak{g},[-,-])$ and $\alpha:V\rightarrow V$ be a linear map. A linear map $T:V\rightarrow \mathfrak{g}$ is called a \textbf{weak $\mathcal{O}$-operator associated to $\mrep{\rho}{V}$ and $\alpha$} if $T$ satisfies
\begin{eqnarray}\mlabel{eq:defi:O-operator1}
    &[T(u),T(v)]=T(\rho(T(u))v-\rho(T(v))u),\;\;\forall u,v\in V,&
\\
\mlabel{eq:defi:O-operator2}
&    PT=T\alpha. &
\end{eqnarray}
If in addition, $\rbrep{\rho}{\alpha}{V}$ is a representation of $(\mathfrak{g},[-,-],P)$, then $T$ is called an \textbf{$\mathcal{O}$-operator associated to $\rbrep{\rho}{\alpha}{V}$}.
\end{defi}

Note that for a Lie algebra $(\frak g,[-,-])$ and a representation $\mrep{\rho}{V}$ of $(\frak g,[-,-])$, a linear map
$T:V\rightarrow \frak g$ satisfying Eq.~\meqref{eq:defi:O-operator1} is called an {\bf $\mathcal O$-operator of $(\frak g,[-,-])$ associated to $\mrep{\rho}{V}$} \mcite{Ku}. The terms relative Rota-Baxter operator and generalized Rota-Baxter operator are also used~\mcite{PBG,Uc}.

\begin{ex}\mlabel{ex:O-operator}
       Let $(\mathfrak{g},[-,-],P)$ be a  Rota-Baxter Lie algebra of
weight zero. Then $P$ is an $\mathcal{O}$-operator of
$(\mathfrak{g},[-,-],P)$ associated to the adjoint representation
$\rbrep{\mathrm{ad}}{P}{\mathfrak{g}}$.
\end{ex}

Theorem \mref{thm:solution} can be rewritten in terms of $\mathcal{O}$-operators as follows.

\begin{cor}\mlabel{cor:solution}
Let $(\mathfrak{g},[-,-],P)$ be a Rota-Baxter Lie algebra and $r\in \mathfrak{g}\otimes\mathfrak{g}$ be antisymmetric. Let $Q:\mathfrak{g}\rightarrow\mathfrak{g}$ be a linear map. Then $r$ is a solution of the $Q$-admissible CYBE in $(\mathfrak{g},[-,-],P)$ if and only if $T_r$ is a weak $\mathcal{O}$-operator associated to
$\mrep{\mathrm{ad}^{*}}{\mathfrak{g}^{*}}$
and $Q^{*}$. If in addition, $(\mathfrak{g},[-,-],P)$ is a $Q$-admissible Rota-Baxter Lie algebra, then $r$ is a solution of the $Q$-admissible CYBE in $(\mathfrak{g},[-,-],P)$ if and only if $T_r$ is an $\mathcal{O}$-operator associated to the representation
$\rbrep{\mathrm{ad}^{*}}{Q^{*}}{\mathfrak{g}^{*}}$.
\end{cor}

On the other hand, an $\mathcal O$-operator of a Lie algebra gives rise to a solution of
the CYBE in the semi-direct product Lie algebra  as follows.

\begin{lem}\mlabel{lem:rt}{\rm \mcite{Bai2007}}
Let $(\frak g,[-,-])$ be a Lie algebra and $\mrep{\rho}{V}$ be a representation. Let
$T: V\rightarrow \frak g$ be a linear map which is identified as an element in
$(\frak g\ltimes_{\rho^*} V^*)\otimes (\frak g\ltimes_{\rho^*} V^*)$
through ${\rm Hom}(V,\frak g)\cong  \frak g\otimes V^* \subseteq (\frak g\ltimes_{\rho^*} V^*)\otimes (\frak g\ltimes_{\rho^*} V^*)$.
Then $=T-\tau(T)$
is an antisymmetric solution of the CYBE
in the Lie algebra $\frak g\ltimes_{\rho^*} V^*$ if and only if $T$
is an $\mathcal O$-operator of $(\frak g,[-,-])$ associated to $(V,\rho)$.
\end{lem}

In order to extend the above construction to the context of Rota-Baxter Lie algebras,
we consider the admissibility of linear maps to the semi-direct products of Rota-Baxter Lie algebras.

\begin{thm}\mlabel{thm:adm SD}
    Let $(\mathfrak{g},[-,-],P)$ be a Rota-Baxter Lie algebra  of weight $\lambda$, and let $\mrep{\rho}{V}$ be a representation of $(\mathfrak{g},[-,-])$. Let $Q:\mathfrak{g}\rightarrow \mathfrak{g}$ and $\alpha,\beta:V\rightarrow V$ be linear maps. Then the following conditions are equivalent.
    \begin{enumerate}
        \item There is a Rota-Baxter Lie algebra $(\mathfrak{g}\ltimes_{\rho}V,P+\alpha)$ such that the linear map $Q+\beta$ on $\mathfrak{g}\oplus V$ is admissible to $(\mathfrak{g}\ltimes_{\rho}V,P+\alpha)$.
        \item There is a Rota-Baxter Lie algebra $(\mathfrak{g}\ltimes_{\rho^{*}}V^{*},P+\beta^{*})$ such that the linear map $Q+\alpha^{*}$ on $\mathfrak{g}\oplus V^{*}$ is admissible to $(\mathfrak{g}\ltimes_{\rho^{*}}V^{*}, P+\beta^{*})$.
        \item The following conditions are satisfied:
        \begin{itemize}
            \item [(i)] $\rbrep{\rho}{\alpha}{V}$ is a representation of $(\mathfrak{g},[-,-],P)$, that is, Eq.~\meqref{eq:defi:rep RB1} holds;
            \item [(ii)] $Q$ is admissible to $(\mathfrak{g},[-,-],P)$, that is, Eq.~\meqref{eq:cor:adm1} holds;
            \item [(iii)] $\beta$ is admissible to $(\mathfrak{g},[-,-],P)$  on $\mrep{\rho}{V}$, that is, Eq.~\meqref{eq:lem:dual map1} holds;
            \item [(iv)] The following equation holds:
            \begin{equation}\mlabel{eq:thm:adm SD1}
                \beta(\rho(x)\alpha(v))=\beta(\rho(Q(x))v)+\rho(Q(x))\alpha(v)+\lambda\rho(Q(x))v,\;\;\forall x\in\mathfrak{g},v\in V.
            \end{equation}
        \end{itemize}
    \end{enumerate}
\end{thm}
\begin{proof}
 The proof follows the same argument as in the proof of  \cite[Theorem 4.20]{BGM}.
    \end{proof}

\begin{cor}\mlabel{cor:adm SD}
Let $(\mathfrak{g},[-,-],P)$ be a Rota-Baxter Lie algebra  of weight $\lambda$ and $\mrep{\rho}{V}$ be a representation of the Lie algebra $(\mathfrak{g},[-,-])$. Let $\alpha:V\rightarrow V$ be a linear map.  Then the following conditions are equivalent.
\begin{enumerate}
    \item \mlabel{it:1a}
    $\rbrep{\rho}{\alpha}{V}$ is a representation of $(\mathfrak{g},[-,-],P)$.
    \item\mlabel{it:1b}  $\alpha^{*}$ is admissible to $(\mathfrak{g}\ltimes_{\rho^{*}}V^{*},P)$.
        \item \mlabel{it:1c} $-\lambda\mathrm{id}_{\mathfrak{g}}+\alpha^{*}$ is admissible to $(\mathfrak{g}\ltimes_{\rho^{*}}V^{*},P-\lambda\mathrm{id}_{V^{*}})$.
        \item \mlabel{it:1d} $-P-\lambda {\rm id}_{\frak g}+\alpha^*$ is admissible to $(\mathfrak{g}\ltimes_{\rho^{*}}V^{*},P-\alpha^*-\lambda\mathrm{id}_{V^{*}})$.
\end{enumerate}
\end{cor}

\begin{proof}
  Suppose that Item (\mref{it:1a}) holds. Then by Proposition~\mref{pro:admissibility} and Corollary~\mref{cor:adm}, $Q$ is admissible to $(\mathfrak{g},[-,-],P)$ and
 $\beta$ is admissible to $(\mathfrak{g},[-,-],P)$  on $\mrep{\rho}{V}$ in the cases when $Q=0,\beta=0$,  or $Q=-\lambda\mathrm{id}_{\mathfrak{g}},\beta=-\lambda\mathrm{id}_{V}$, or  $Q=-P-\lambda\mathrm{id}_{\mathfrak{g}},\beta=-\alpha-\lambda\mathrm{id}_{V}$. Moreover, in these cases, Eq.~\meqref{eq:thm:adm SD1} holds. Hence by Theorem ~\mref{thm:adm SD}, Items (\mref{it:1b})-(\mref{it:1d}) follow since they correspond to these cases respectively.

 Conversely, suppose that any of Items~(\mref{it:1b})-(\mref{it:1d}) holds, then by Theorem ~\mref{thm:adm SD}, $\rbrep{\rho}{\alpha}{V}$ is a representation of $(\mathfrak{g},[-,-],P)$, that is, Item~(\mref{it:1a}) holds.
   \end{proof}

In the following, we apply $\calo$-operators to the constructions of antisymmetric solutions of the admissible CYBE, and of Rota-Baxter Lie bialgebras.

\begin{thm}\mlabel{thm:bialgebra}
    Let $(\mathfrak{g},[-,-]_{\mathfrak{g}},P)$ be a Rota-Baxter Lie algebra of weight $\lambda$ and $\mrep{\rho}{V}$ be a representation of $(\mathfrak{g},[-,-]_{\mathfrak{g}})$. Let $\beta:V\rightarrow V$ be a linear map which is admissible to $(\mathfrak{g},[-,-]_{\mathfrak{g}},P)$ on $\mrep{\rho}{V}$. Let $Q:\mathfrak{g}\rightarrow\mathfrak{g}$, $\alpha:V\rightarrow V$ and $T:V\rightarrow \mathfrak{g}$ be linear maps.
    \begin{enumerate}
        \item \mlabel{it:111}  $r=T-\tau(T)$ is an antisymmetric solution of the $(Q+\alpha^{*})$-admissible CYBE in the Rota-Baxter Lie algebra $(\mathfrak{g}\ltimes_{\rho^{*}}V^{*},P+\beta^{*})$ if and only if $T$ is a weak $\mathcal{O}$-operator associated to $\mrep{\rho}{V}$ and $\alpha$, and satisfies $T\beta=QT$.
        \item \mlabel{it:112} Assume that $\rbrep{\rho}{\alpha}{V}$ is a representation of $(\mathfrak{g},[-,-]_{\mathfrak{g}},P)$. If $T$ is an $\mathcal{O}$-operator associated to $\rbrep{\rho}{\alpha}{V}$ and $T\beta=QT$, then $r=T-\tau(T)$ is an antisymmetric solution of the $(Q+\alpha^{*})$-admissible CYBE in the Rota-Baxter Lie algebra $(\mathfrak{g}\ltimes_{\rho^{*}}V^{*},P+\beta^{*})$. If in addition, $(\mathfrak{g},[-,-]_{\mathfrak{g}},P)$ is $Q$-admissible and Eq.~\meqref{eq:thm:adm SD1} holds such that  the Rota-Baxter algebra $(\mathfrak{g}\ltimes_{\rho^{*}}V^{*},P+\beta^{*})$ is $(Q+\alpha^{*})$-admissible, then there is a Rota-Baxter Lie bialgebra $(\mathfrak{g}\ltimes_{\rho^{*}}V^{*},P+\beta^{*},\delta,Q+\alpha^{*})$ of weight $\lambda$, where the linear map $\delta=\delta_{r}$ is defined by Eq.~\meqref{eq:LieCob} with $r=T-\tau(T)$.
    \end{enumerate}
\end{thm}

\begin{proof}
(\mref{it:111}). It is the same as the proof of \cite[Theorem 4.21 (a)]{BGM}.

(\mref{it:112}). It follows from Item (\mref{it:111}) and Theorem \mref{thm:adm SD}.
    \end{proof}

\begin{cor}\mlabel{cor:bialgebras}
Let $(\mathfrak{g},[-,-],P)$ be a Rota-Baxter Lie algebra of
weight $\lambda$ and $\rbrep{\rho}{\alpha}{V}$ be a representation of
$(\mathfrak{g},[-,-],P)$. Suppose that
$T:V\rightarrow\mathfrak{g}$ is an $\mathcal{O}$-operator of
$(\mathfrak{g},[-,-],P)$ associated to $\rbrep{\rho}{\alpha}{V}$. Then
there are the Rota-Baxter Lie bialgebras
$(\mathfrak{g}\ltimes_{\rho^{*}}V^{*},P,\delta,\alpha^{*})$
$(\mathfrak{g}\ltimes_{\rho^{*}}V^{*},P-\lambda\mathrm{id}_{V^{*}},\delta,-\lambda\mathrm{id}_{\mathfrak{g}}+\alpha^{*})$
and
$(\mathfrak{g}\ltimes_{\rho^{*}}V^{*},P-\alpha^{*}-\lambda\mathrm{id}_{V^{*}},\delta,-P-\lambda\mathrm{id}_{\mathfrak{g}}+\alpha^{*})$,
where the linear map $\delta=\delta_{r}$ is defined by
Eq.~\meqref{eq:LieCob} with $r=T-\tau(T)$.
\end{cor}

\begin{proof}
By Corollary~\mref{cor:adm SD}, the facts that the operator $\alpha^{*}$ (resp. $-\lambda\mathrm{id}_{\mathfrak{g}}+\alpha^{*}$, resp. $-P-\lambda {\rm id}_{\frak g}+\alpha^*$) is
admissible to $(\mathfrak{g}\ltimes_{\rho^{*}}V^{*},P)$ (resp.
$(\mathfrak{g}\ltimes_{\rho^{*}}V^{*},P-\lambda\mathrm{id}_{V^{*}})$, resp. $(\mathfrak{g}\ltimes_{\rho^{*}}V^{*},P-\alpha^*-\lambda\mathrm{id}_{V^{*}})$)
correspond to the case of $Q=0,\beta=0$ (resp. $Q=-\lambda\mathrm{id}_{\mathfrak{g}},\beta=-\lambda\mathrm{id}_{V}$,
resp.
 $Q=-P-\lambda\mathrm{id}_{\mathfrak{g}},\beta=-\alpha-\lambda\mathrm{id}_{V}$)
 in Theorem \mref{thm:adm SD}. Note that in each of these cases, $T\beta=QT$. Hence the
 conclusion follows from Theorem~\mref{thm:bialgebra} \meqref{it:112}.
\end{proof}

To illustrate the construction of Rota-Baxter Lie bialgebras  by
$\mathcal{O}$-operators, we focus on the special case when the
$\calo$-operators are associated to the adjoint representation of
the Rota-Baxter Lie algebra, as given in Example
\mref{ex:O-operator}.

\begin{pro}\mlabel{cor:bialgebra}
Let $(\mathfrak{g},[-,-],P)$ be a Rota-Baxter Lie algebra of weight $\lambda$.
\begin{enumerate}
    \item\mlabel{cor:bialgebra1}  Let $T:\mathfrak{g}\rightarrow\mathfrak{g}$ be an $\mathcal{O}$-operator of $(\mathfrak{g},[-,-],P)$
    associated to the adjoint representation
    $\rbrep{\mathrm{ad}}{P}{\mathfrak{g}}$.
    Suppose that $Q$ is admissible to $(\mathfrak{g},[-,-],P)$ and $TQ=QT$.
    Then there is a Rota-Baxter Lie bialgebra $(\mathfrak{g}\ltimes_{\mathrm{ad}^{*}}\mathfrak{g}^{*},P+Q^{*},\delta_r,Q+P^{*})$, with the linear map $\delta_{r}$ defined by Eq.~\meqref{eq:LieCob} with $r=T-\tau(T)$.

\item \mlabel{it:b2} Let $\lambda= 0$.
        Suppose that $Q:\mathfrak{g}\rightarrow \mathfrak{g}$ is a linear map that is admissible to $(\mathfrak{g},[-,-],P)$ and commutes with $P$. Then there is a  Rota-Baxter
         Lie bialgebra $(\mathfrak{g}\ltimes_{\mathrm{ad}^{*}}\mathfrak{g}^{*},P+Q^{*},\delta_r,Q+P^{*})$ of weight zero, with $\delta_r$ as defined in the last item by letting $T=P$.
        In particular, there are  Rota-Baxter Lie bialgebras $(\mathfrak{g}\ltimes_{\mathrm{ad}^{*}}\mathfrak{g}^{*},P-P^{*},\delta_r,-P+P^{*})$ and $(\mathfrak{g}\ltimes_{\mathrm{ad}^{*}}\mathfrak{g}^{*},P,\delta_r,P^{*})$ of weight zero.
\end{enumerate}
\end{pro}

\begin{proof}
(\mref{cor:bialgebra1}). It follows from
Theorem~\mref{thm:bialgebra}
 (\mref{it:112}) in the case that $\rho={\rm
ad}$, $\alpha=P$, $\beta=Q$.

(\mref{it:b2}). By Example \mref{ex:O-operator}, $P$ is an
$\mathcal{O}$-operator of $(\mathfrak{g},[-,-],P)$ associated to
$(\mathfrak{g}, {\rm ad}, P)$. Then by Item
(\mref{cor:bialgebra1}), the first conclusion follows by letting
$T=P$.
Furthermore, note that both $Q=-P$ and $Q=0$ are admissible to
$(\mathfrak{g},[-,-],P)$ and commute with $P$. Then the second
conclusion holds.
\end{proof}

\begin{defi}\mlabel{defi:RB-pre-Lie alg}
A \textbf{Rota-Baxter pre-Lie algebra of weight $\lambda$} is a
triple $(A,\circ,P)$, such that $(A,\circ)$ is a pre-Lie algebra,
and $P:A\rightarrow A$ is a \textbf{Rota-Baxter operator of weight
$\lambda$} on $(A,\circ)$, that is, $P$ satisfies
\begin{equation}
    P(x)\circ P(y)=P(P(x)\circ y)+P(x\circ P(y))+\lambda P(x\circ y),\;\;\forall x,y\in A.
\end{equation}
\end{defi}

By a direct verification, we obtain
\begin{pro}\mlabel{pro:RBPLie}
Let $(A,\circ,P)$ be a Rota-Baxter pre-Lie algebra
of weight $\lambda$. Then the following conclusions hold.
\begin{enumerate}
\item $(\mathfrak{g}(A),[-,-],P)$ is a Rota-Baxter Lie algebra of
weight $\lambda$, which is called the {\bf sub-adjacent
Rota-Baxter Lie algebra} of $(A,\circ,P)$.

\item
$\rbrep{\mathcal{L}_{\circ}}{P}{A}$ is a representation of the
Rota-Baxter Lie algebra $(\mathfrak{g}(A),[-,-],P)$.

\item \mlabel{it:33} The identity map $\mathrm{id}_{A}$ on $A$ is an $\mathcal{O}$-operator on the Rota-Baxter Lie algebra
$(\mathfrak{g}(A),$ $[-,-],P)$ associated to $\rbrep{\mathcal{L}_{\circ}}{P}{A}$.
\end{enumerate}
\end{pro}
On the other hand, the following conclusion is also easy to check.

\begin{pro}\mlabel{pro:induced2} Let  $(\mathfrak{g},[-,-]_{\mathfrak{g}},P)$ be a Rota-Baxter Lie
algebra of weight $\lambda$ and $(V,\rho,\alpha)$ be a
representation of $(\mathfrak{g},[-,-]_{\mathfrak{g}},P)$. Let
$T:V\rightarrow \frak g$ be an $\mathcal O$-operator associated to
$(V,\rho,\alpha)$. Then there exists a Rota-Baxter pre-Lie algebra
structure $(V,\circ, \alpha)$ on $V$, with $\circ$ given by
\begin{equation}
u\circ v:=\rho(T(u))v,\;\;\forall u,v\in V.
\end{equation}
In particular, if $(\mathfrak{g},[-,-],P)$ is a  Rota-Baxter Lie
algebra of weight zero and $(\frak g, \circ)$ is the induced
pre-Lie algebra, then $P$ is  a Rota-Baxter operator of weight
zero on $(\mathfrak{g},\circ)$.
\end{pro}

There is a simple construction of Rota-Baxter Lie bialgebras from Rota-Baxter pre-Lie algebras.

\begin{pro}\mlabel{pro:L}
    Let $(A,\circ,P)$ be a Rota-Baxter pre-Lie algebra of weight $\lambda$, and let the sub-adjacent Rota-Baxter Lie algebra  be $(\mathfrak{g}(A),[-,-],P)$.
Let $\{e_1,\cdots,e_n\}$ be a basis of $A$,
$\{e_1^*,\cdots,e_n^*\}$ be the dual basis and $r=
\sum\limits_{i=1}^{n}e_{i}\otimes e^{*}_{i}-e^{*}_{i}\otimes e_{i}$.
Define the linear map $\delta=\delta_{r}$ by Eq.~\meqref{eq:LieCob}.
Then there are Rota-Baxter Lie bialgebras
        $(\mathfrak{g}(A)\ltimes_{\mathcal{L}^{*}_{\circ}}A^{*},P,\delta,P^{*})$,
        $(\mathfrak{g}(A)\ltimes_{\mathcal{L}^{*}_{\circ}}A^{*},P-\lambda\mathrm{id}_{A^{*}},\delta,-\lambda\mathrm{id}_{A}+P^{*})$,
         and
         $(\mathfrak{g}(A)\ltimes_{\mathcal{L}^{*}_{\circ}}A^{*},P-\lambda\mathrm{id}_{A^{*}}-P^{*},\delta,-P-\lambda\mathrm{id}_{A}+P^{*})$.
\end{pro}

\begin{proof}
   By Proposition~\mref{pro:RBPLie} (\mref{it:33}),  $T=\mathrm{id}_{A}$ is an $\mathcal{O}$-operator of the Rota-Baxter Lie algebra  $(\mathfrak{g}(A)$, $[-,-]$, $P)$ associated to $\rbrep{\mathcal{L}_{\circ}}{P}{A}$.
Note that $r=\mathrm{id}_{A}-\tau(\mathrm{id}_{A})=
    \sum\limits_{i=1}^{n}e_{i}\otimes e^{*}_{i}-e^{*}_{i}\otimes
    e_{i}$. Then the conclusion follows from
    Corollary~\mref{cor:bialgebras}.
    \end{proof}

To complete the paper, we construct 
\spec \ldendbs from $\mathcal O$-operators on Rota-Baxter Lie algebras of weight zero and from Rota-Baxter pre-Lie algebras.

\begin{pro}\mlabel{pro:cons-L}
Let $(\mathfrak{g},[-,-]_{\mathfrak{g}},P)$ be a Rota-Baxter Lie
algebra of weight zero. Let $Q:\frak g\rightarrow \frak g$ be a
linear map which is admissible to
$(\mathfrak{g},[-,-]_{\mathfrak{g}},P)$. Let $\rbrep{\rho}{\alpha}{V}$ be
a representation of $(\mathfrak{g},[-,-]_{\mathfrak{g}},P)$ and
$\beta:V\rightarrow V$ be a linear map which is admissible to
$(\mathfrak{g},[-,-]_{\mathfrak{g}},P)$ on $\mrep{\rho}{V}$. Suppose
Eq.~\meqref{eq:thm:adm SD1} holds. If $T$ is an
$\mathcal{O}$-operator associated to $\rbrep{\rho}{\alpha}{V}$ and
$T\beta=QT$, then there is a \spec \ldendb
        $(\mathfrak{g}\ltimes_{\rho^{*}}V^{*},\triangleright,\triangleleft,\copa,\copb)$, where
\begin{eqnarray*}
&&a\triangleleft b=-(Q+\alpha^{*})([a,b]), a\triangleright b=[(P+\beta^{*})a,b]+(Q+\alpha^{*})[a,b],\\
&&\copb(a)=-\delta((P+\beta^{*})a),
\copa(a)=((Q+\alpha^{*})\otimes\mathrm{id})\delta(a)+\delta((P+\beta^{*})a),
\end{eqnarray*}
for all $a,b\in \mathfrak{g}\ltimes_{\rho^{*}}V^{*}$. Here $[-,-]$ is defined
by the Lie algebra $\mathfrak{g}\ltimes_{\rho^{*}}V^{*}$ and the
linear map $\delta=\delta_{r}$ is defined by Eq.~\meqref{eq:LieCob} with $r=T-\tau(T)$.
\end{pro}

\begin{proof}
It follows from Theorem~\mref{thm:bialgebra} (\mref{it:112}) and
Proposition~\mref{pro:condition}.
\end{proof}

\begin{cor}\mlabel{cor:cons1}
 Let $(\mathfrak{g},[-,-]_{\mathfrak{g}},P)$ be a Rota-Baxter
Lie algebra of weight zero.
Suppose that $Q:\mathfrak{g}\rightarrow \mathfrak{g}$ is a linear map that is admissible to $(\mathfrak{g},[-,-]_{\mathfrak{g}},P)$ and commutes with $P$.
Let $[-,-]$ be the bracket on the Lie algebra $\mathfrak{g}\ltimes_{{\rm ad}^{*}}\frak g^{*}$, and let $\delta=\delta_{r}$ be the linear map defined by Eq.~\meqref{eq:LieCob} with $r=P-\tau(P)$. Then there is a \spec \ldendb
        $(\mathfrak{g}\ltimes_{{\rm ad}^{*}}\frak g^{*},\triangleright,\triangleleft,\copa,\copb)$, where
\vspace{-.2cm}
\begin{eqnarray}
&&a\triangleleft b=-(Q+P^{*})([a,b]),\ \  a\triangleright b=[(P+Q^{*})a,b]+(Q+P^{*})[a,b],\mlabel{eq:sl1}\\
&&\copb(a)=-\delta((P+Q^{*})a), \ \
\copa(a)=((Q+P^{*})\otimes\mathrm{id})\delta(a)+\delta((P+Q^{*})a),\mlabel{eq:sl2}
\vspace{-.2cm}
\end{eqnarray}
for all $a,b\in \mathfrak{g}\ltimes_{{\rm ad}^{*}}\frak g^{*}$. In particular, there are two special
\ldendbs
        $(\mathfrak{g}\ltimes_{{\rm ad}^{*}}\frak
        g^{*},\triangleright_1,\triangleleft_1,\copa_1,\copb_1)$
        and
        $(\mathfrak{g}\ltimes_{{\rm ad}^{*}}\frak g^{*},\triangleright_2,\triangleleft_2,\copa_2,\copb_2)$, where
  \begin{eqnarray}
        &&a\triangleleft_1 b=-(-P+P^{*})([a,b]), a\triangleright_1 b=[(P-P^{*})a,b]+(-P+P^{*})[a,b],\mlabel{eq:sl3}\\
        &&\copb_1(a)=-\delta((P-P^{*})a),
        \copa_1(a)=((-P+P^{*})\otimes\mathrm{id})\delta(a)+\delta((P-P^{*})a),\mlabel{eq:sl4}\\
        &&a\triangleleft_{2} b=-P^{*}([a,b]), a\triangleright_{2} b=[P(a),b]+P^{*}([a,b]),\mlabel{eq:sl5}\\
        &&\copb_2(a)=-\delta(P(a)),
        \copa_2(a)=(P^{*}\otimes\mathrm{id})\delta(a)+\delta(P(a)).\mlabel{eq:sl6}
        \end{eqnarray}
\end{cor}

\begin{proof}
It follows from Propositions~\mref{cor:bialgebra} (\mref{it:b2}) and
\mref{pro:cons-L}.
\end{proof}

\begin{cor}\mlabel{cor:cons2}
 Let $(A,\circ,P)$ be a Rota-Baxter pre-Lie algebra of weight zero with its sub-adjacent Rota-Baxter Lie algebra $(\mathfrak{g}(A),[-,-],P)$. Then there are two special
\ldendbs $(\mathfrak{g}(A)\ltimes_{\mathcal{L}_{\circ}^{*}}A^{*},\triangleright_1,\triangleleft_1,\copa_1,\copb_1)$
and
$(\mathfrak{g}(A)\ltimes_{\mathcal{L}_{\circ}^{*}}A^{*},\triangleright_2,\triangleleft_2,\copa_2,\copb_2)$,
where $\triangleright_1,\triangleleft_1,\copa_1,\copb_1$ and
$\triangleright_2,\triangleleft_2,\copa_2,\copb_2$ are defined by
Eqs.~\meqref{eq:sl3}-\meqref{eq:sl6} respectively, $[-,-]$ is
defined by the Lie algebra
$\mathfrak{g}(A)\ltimes_{\mathcal{L}_{\circ}^{*}}A^{*}$ and
 the linear map $\delta=\delta_{r}$ is defined by Eq.~\meqref{eq:LieCob} with $r=
\sum\limits_{i=1}^{n}e_{i}\otimes e^{*}_{i}-e^{*}_{i}\otimes
e_{i}$, where $\{e_1,\cdots,e_n\}$ is a basis of $A$ and
$\{e_1^*,\cdots,e_n^*\}$ is the dual basis.
\end{cor}

\begin{proof}
It follows from Propositions \mref{pro:L} and \mref{pro:condition}.
\end{proof}

We end the paper by showing that, starting with any Rota-Baxter Lie algebra  of weight zero, the various Rota-Baxter Lie bialgebras from it allow us to construct a family of \spec \ldendbs.

\begin{ex}
Let $(\mathfrak{g},[-,-]_{\mathfrak{g}},P)$ be a Rota-Baxter
Lie algebra of weight zero. Let $(\mathfrak{g},\circ)$ be the induced pre-Lie algebra, whose sub-adjacent Lie algebra is denoted by $(\mathfrak{g}',[-,-]'_{\mathfrak{g}})$ or simply $\mathfrak{g}'$, that is, $[-,-]_{\frak g}'$ is defined by
$$[x,y]_{\frak g}'=[P(x),y]_{\frak g}+[x,P(y)]_{\frak g},\;\;\forall x,y\in \frak g.$$
\begin{enumerate}
    \item \mlabel{it:aaa} By Corollary~\mref{cor:cons1}, there are special
    \ldendbs
    $(\mathfrak{g}\ltimes_{{\rm ad}^{*}}\frak
    g^{*},\triangleright_1,\triangleleft_1,\copa_1,\copb_1)$
    and
    $(\mathfrak{g}\ltimes_{{\rm ad}^{*}}\frak g^{*},\triangleright_2,\triangleleft_2,\copa_2,\copb_2)$,
 where $\triangleright_1,\triangleleft_1,\copa_1,\copb_1$ and $\triangleright_2,\triangleleft_2,\copa_2,\copb_2$ are
defined by Eqs.~\meqref{eq:sl3}-\meqref{eq:sl6} respectively, $[-,-]$ is defined
by the Lie algebra $\mathfrak{g}\ltimes_{{\rm ad}^{*}}\frak g^{*}$ and the linear map $\delta=\delta_{r}$ is defined by Eq.~\meqref{eq:LieCob}
with $r=P-\tau(P)$.

    \item~\mlabel{it:bbb} By Proposition~\mref{pro:induced2}, $P$ is a Rota-Baxter operator of weight zero on $(\frak g,\circ)$. Then by Corollary~\mref{cor:cons2}, there are two special
    \ldendbs
    $(\mathfrak{g}'\ltimes_{\mathcal{L}_{\circ}^{*}}\mathfrak{g}^{*},\triangleright_3,\triangleleft_3,\copa_3,\copb_3)$
    and
    $(\mathfrak{g}'\ltimes_{\mathcal{L}_{\circ}^{*}}\mathfrak{g}^{*}$, $\triangleright_4$, $\triangleleft_4$, $\copa_4$, $\copb_4)$, where $\triangleright_3,\triangleleft_3,\copa_3,\copb_3$ and $\triangleright_4,\triangleleft_4,\copa_4,\copb_4$ are
defined by Eqs.~\meqref{eq:sl3}-\meqref{eq:sl6} respectively, $[-,-]$ is defined by the Lie algebra $\mathfrak{g}'\ltimes_{\mathcal{L}_{\circ}^{*}}\mathfrak{g}^{*}$ and
 the linear map $\delta=\delta_{r}$ is defined by Eq.~\meqref{eq:LieCob} with
  $r=\sum\limits_{i=1}^{n}e_{i}\otimes e^{*}_{i}-e^{*}_{i}\otimes
    e_{i}$, where $\{e_1,\cdots,e_n\}$ is a basis of $\mathfrak{g}$ and
    $\{e_1^*,\cdots,e_n^*\}$ is the dual basis.
    \item By Proposition~\mref{pro:RBPLie}, $P$ is a Rota-Baxter operator of weight zero on the Lie algebra $(\frak g',[-,-]_{\frak g'})$. Hence for
this Rota-Baxter Lie algebra $(\frak g', [-,-]_{\frak g'}, P)$ of weight zero, we can repeat Items~\meqref{it:aaa} and \meqref{it:bbb} to get
four \spec \ldendbs. Further repeating the same procedure gives rise to a series of \spec \ldendbs.
\end{enumerate}
\end{ex}

\noindent
 {\bf Acknowledgments.}  This work is supported by
 National Natural Science Foundation of China (Grant No.  11931009), the Fundamental Research Funds for the Central
Universities and Nankai Zhide Foundation.

\end{document}